\documentclass[10pt]{article}
\newtheorem{prelemmaa}{{\bf LEMMA}}

\newtheorem{prelem}{{\bf THEOREM}}

\newtheorem{preque}{{\bf QUESTION}}

\newtheorem{theorem}{THEOREM}

\newtheorem{prelemma}{LEMMA}

\newenvironment{lemma}{\begin{prelemma}{\hspace{-0.5
               em}}}{\end{prelemma}}
\newtheorem{preproof}{{\bf PROOF.}}

\newenvironment{proof}[1]{\begin{preproof}{\rm
               #1}\hfill{\rule[-0.5mm]{2mm}{2mm}}}{\end{preproof}}

\newtheorem{preproposition}{{PROPOSITION}}

\newenvironment{proposition}{\begin{preproposition}{\hspace{-0.5
                em}}}{\end{preproposition}}
\newtheorem{preremark}{REMARK}

\newtheorem{precorollary}{{COROLLARY}}

\newtheorem{predefinition}{DEFINITION}

\newtheorem{preexample}{EXAMPLE}

\newtheorem{preconjecture}{{CONJECTURE}}

\newtheorem{pretheo}{{\bf THEOREM}}

\def\newpic#1{}
\date{}
\begin{document}
\title{\bf Super-simple $2$-$(v,5,1)$ directed designs and their smallest defining sets}
\author{
{\sc Nasrin Soltankhah\footnote{Corresponding author. E-mail
address: soltan@alzahra.ac.ir, soltankhah.n@gmail.com}\  AND Farzane
Amirzade  }
 \\[5mm]
Department of Mathematics\\ Alzahra University  \\
Vanak Square 19834 \ Tehran, I.R. Iran }
%
\maketitle
\begin{abstract}
In this paper we investigate the spectrum of super-simple $2$-$(v,5,1)$ directed designs (or simply super-simple $2$-$(v,5,1)$DDs) and also the size of their smallest defining sets.

We show that for all  $v\equiv1,5\ ({\rm mod}\ 10)$ except $v=5,15$ there exists
a super-simple $(v,5,1)DD$. Also for these parameters, except possibly $v=11,91$, there exists a super-simple $2$-$(v,5,1)$DD whose smallest defining sets have at least a half of the blocks.
 \end{abstract}
%
\hspace*{-2.7mm} {\bf KEYWORDS:} {  \sf super simple directed designs, smallest
defining set }
%
\section{Introduction} 
\setcounter{preproposition}{0} \setcounter{precorollary}{0}
\setcounter{prelemma}{0} \setcounter{preexample}{0}
Let $0<t\leq k\leq v$ and $\lambda>0$ be integers. A $t$-$(v,k,\lambda)$ directed design (or simply a
$t$-$(v,k,\lambda)$DD) is a pair $(V,\mathcal{B})$, where $V$ is
a set of $v$ elements, called points, and $\mathcal{B}$ is a collection of  ordered $\it
k-$tuples of distinct elements of $V$, called blocks, with the property that
every ordered $t$-tuple of distinct elements of $V$ occurs in
exactly $\lambda$ blocks (as a subsequence). We say that a $t$-tuple appears in a
$\it k$-tuple if its components appear in that $k$-tuple as a set,
and they appear with the same order. For example the following
base blocks form a $2$-$(21,5,1)$DD$, \mathcal{D}$.
$$
\begin{array}{l}
(0,\ 1,\ 4,\ 14,\ 16),\ (1,\ 0,\ 18,\ 8,\ 6)\\
\end{array}
$$
Here, for example, the 5-tuple $(0,\ 1,\ 4,\ 14,\ 16)$ contains the ordered pairs (0,1),
(0,4), (0,14), (0,16), (1,4), (1,14), (1,16), (4,14), (4,16), (14,16).

In this paper, we extensively use the concept of ``trade'' defined as follows.

 A $(v,k,t)$ directed trade of volume $s$ consists of two
disjoint collections
 $T_1$ and $T_2$, each of
$s$ blocks, such that every $t$-tuple of distinct elements of $V$ is
covered by precisely the same number of blocks of $T_1$ as of $T_2$.
Such a directed trade is usually denoted by $T=T_1-T_2$. Blocks in
$T_1(T_2)$ are called the positive (respectively, negative) blocks
of T. If ${\mathcal{D}}=(V,\mathcal{B})$ is a directed design, and
if $T_1\subseteq \mathcal{B}$, we say that $\mathcal{D}$ contains
the directed trade T. For example the $2$-$(21,5,1)$DD$, \mathcal{D}$
above contains the following directed trade:
$$
\begin{array}{ccccccc}
\frac{\ \sc T_1\ }{\ }              &      & \frac{\ \sc T_2\ }{\ }  \\
0\ 1\ 4\ 14\ 16\                    &      & 1\ 0\ 4\ 14\ 16\\
1\ 0\ 18\ 8\ 6\                     &      & 0\ 1\ 18\ 8\ 6\
\end{array}
$$
A set of blocks which is a subset of a unique $t$-$(v,k,\lambda)$
directed design $\mathcal{D}$, is called a defining set of
$\mathcal{D}$. In other words for a given a $t$-$(v,k,\lambda)$
directed design $\mathcal{D}$, a subset of the blocks of
$\mathcal{D}$ that occurs in no other $t$-$(v,k,\lambda)$ directed
design is called a defining set of $\mathcal{D}$.

Defining sets for directed designs are strongly related to trades. This relation
is illustrated by the following result.
\begin{proposition}~{\rm \cite{MMS}} Let $\mathcal{D}=(V,\mathcal{B})$ be a $t$-$(v,k,\lambda)$ directed design and
let $S\subseteq
\mathcal{B}$, then $S$ is a defining set of $\mathcal{D}$ if only if
$S$ contains a block of every $(v,k,t)$ directed trade $T=T_1-T_2$
such that $T$ is contained in $\mathcal{D}$.
\end{proposition}
Each defining set of a $t$-$(v,k,\lambda)$DD$, \mathcal{D}$ contains at
least one block in every trade in $\mathcal{D}$. In particular, if
$\mathcal{D}$ contains $m$ mutually disjoint directed trades then
the smallest defining set of $\mathcal{D}$ must contain at least $m$
blocks.

The concept of directed trades and defining sets for directed
designs were investigated in articles $\cite{MMS, trade}$.

To construct designs we use a special type of directed trade, called a
cyclical trade, defined as follows.

Let $T=T_1-T_2$ be a $(v,5,2)$ directed trade of volume $s$, where
$T_1$ contains blocks $b_1,\dots,b_s$ such that each pair of
consecutive 5-tuples (blocks) of $T_1$, $b_i,b_{i+1}$ $i=1,\dots,s\
({\rm mod}\ s)$ is a trade of volume 2. Therefore if a directed
design $\mathcal{D}$ contains $T_1$, then any defining set for
$\mathcal{D}$ must contain at least $[\frac{s+1}{2}]$ blocks of
$T_1$.

A $2$-$(v,k,\lambda)$ directed design (or simply a
$2$-$(v,k,\lambda)$DD) is called simple if its underlying
$2$-$(v,k,2\lambda)$-BIBD contains no repeated blocks. A
$2$-$(v,k,\lambda)$DD   is super-simple if its underlying
$2$-$(v,k,2\lambda)$-BIBD is super-simple, that is, any two blocks
of the BIBD intersect in at most two points.

The concept of super-simple designs was introduced by Mullin and Gronau
$\cite{mullin}$. There are known results for the existence of
super-simple designs, especially for the existence of super-simple
$(v,k,\lambda)$-BIBDs. When $k=5$  the necessary conditions for
super-simple $(v,k,\lambda)$-BIBDs are known to be sufficient for
$2\leq \lambda \leq 5$ with few possible exceptions. These known
results can be found in articles $\cite{pentagon, Chen and Wei2007,
Chen and Wei2006, pentagon2004}$.

When $k=3$, a super-simple design is just a simple design. In
$\cite{Quinn}$ Grannell, Griggs and Quinn have shown that for each
admissible value of $v$, there exists a simple $2$-$(v,3,1)$DD whose
smallest defining sets have at least a half of the blocks. In
$\cite{farzane}$ Amirzade and Soltankhah have proved
 that for all $v\equiv1\ ({\rm mod}\ 3)$ there exists a super-simple $2$-$(v,4,1)$DD whose smallest defining sets have
  at least a half of blocks. The necessary condition for the existence of a  super-simple $2$-$(v,5,1)$ directed design
  is that $v\equiv1,5\ ({\rm mod}\ 10)$ and $v\neq5,15$.

 In this paper we show that this necessary condition is
  sufficient. Also we show that for these parameters, except possibly $v=11,91$, there exists a
  super-simple $2$-$(v,5,1)$DD whose smallest defining sets have at least a half of the blocks.

In other words we are interested in the quantity
\begin{center}
$\large f=\frac{{\rm number}\ {\rm of}\ 5-{\rm tuples}\ {\rm in}\
{\rm a}\ {\rm smallest}\ {\rm defining}\ {\rm set}\ {\rm in}\
\mathcal{D}} {{\rm number}\ {\rm of}\ 5-{\rm tuples}\ {\rm in}\
\mathcal{D}}.$
\end{center}
We show that, there exists a super-simple $2$-$(v,5,1)$DD, $\mathcal{D}$
with $f\geq\frac{1}{2}$.

The proofs in this paper use various types of combinatorial objects.
The definitions of these objects are either given in this section or
can be found in the related references.

A pairwise balanced design of order $v$ with block sizes $k\in K$ or PBD$(v,K,\lambda)$ is a pair $(V,\mathcal{B})$, where $V$ is a
$v$-set, and $\mathcal{B}$ is a collection of subsets (called blocks) of $V$ such that if $B\in\mathcal{B}$ then $|B|\in K$ and every pair of distinct elements of $V$ appears in
precisely $\lambda$ blocks. If $\lambda=1$, PBD$(v,K,1)$ is denoted PBD$(v,K)$.

A group divisible design of order $v$ with block sizes $k\in K$ or $(K,\lambda)$-GDD of type ${g_1}^{u_1}{g_2}^{u_2}\dots {g_N}^{u_N}$, where $u_1,u_2,\dots,u_N$ are non-negative integers, is a triple $(V,\mathcal{G},\mathcal{B})$, where $V$ is a
$v$-set that is partitioned into parts (called groups) of sizes $g_1,g_2,\dots, g_N$, and $\mathcal{B}$ is a collection of subsets (called blocks) of $V$ such that if $B\in\mathcal{B}$ then $|B|\in K$ and every pair of distinct elements of $V$ appears in
precisely $\lambda$ blocks or one group but not in both. If $\lambda=1$, $(K,1)$-GDD is denoted $K$-GDD.

We can
delete one point from a PBD$(v,K)$ to form a $K$-GDD of order $v-1$ of type ${g_1}^{u_1}{g_2}^{u_2}\dots {g_N}^{u_N}$, where $u_1,u_2\dots u_N$ are non-negative
integers and for all $i=1,2,\dots, N$, $g_i=k_i-1;\ k_i\in K$.

A transversal design TD$_{\lambda} (k,n)$ is a $(K,\lambda)$-GDD of type $n^k$. If $\lambda=1$, TD$_{1} (k,n)$ is denoted TD$(k,n)$.

A directed group divisible design  $(K,\lambda)$-DGDD is a group divisible design $GDD$ in which every block is ordered and each ordered pair formed from distinct elements of different groups occurs in exactly $\lambda$ blocks. A $(K,\lambda)$-DGDD of type $({g_1}^{u_1}{g_2}^{u_2}\dots{g_N}^{u_N})$ is super-simple if its underlying  $(K,2\lambda)$-GDD of type $({g_1}^{u_1}{g_2}^{u_2}\dots {g_N}^{u_N})$ is super-simple.

%

\section{Super-simple directed group divisible designs with block size 5 and index 1}  

The constructions used in this paper will combine both direct and recursive methods. For our direct constructions,
we shall adopt the standard approach of using finite abelian groups to generate the set of blocks for any given super-simple DGDDs or super-simple $2$-$(v,5,1)$DDs. That is, instead of listing all of the blocks, we give the element set $V$, plus a set of base blocks, and generate the other blocks by an additive group $G$. For $g=|G|$, the notation $(+t\ {\rm mod}\ g)$ means the base blocks should be developed by adding $0,\ t,\ 2t,\ ,\dots, g-t\ ({\rm mod}\ g)$ to them.

In this section we construct some super-simple DGDDs with $f\geq\frac{1}{2}$ that we will use in our main result.
These $(k,\lambda)$-DGDDs are super-simple because their underlying $(k,2\lambda)$-GDDs are super-simple,
(see $\cite{pentagon, HSPS, pentagon2004}$).

Here we will construct super-simple DGDDs with block size 5 and
index 1 of type $(g^u)$ for $(g,u)=(5,5),\ (5,7),\ (5,9),\ (6,5),\
(6,6), (10,5),\ (4,5),\ (4,6),\ (4,10),\ (4,11),\ (4,16),\ (2,6),\
(10,6)$ and super-simple DGDDs of type $(g^um^1)$ for
$(g,u,m)=(6,5,8),\ (4,8,6),\ (1,20,3),\ (1,16,5)$.

Now we can construct a super-simple DGDD of type $(5^5)$ on the set
$V=\{0,1,2,3,4\} \times Z_5$ with groups $\{(i , 0),\ (i , 1),\ (i ,
2),\ (i , 3),\ (i , 4)\};\ i=0,1,2,3,4$ by developing the second
coordinate of the following 10 base blocks modulo 5 or simply by $(-
, +1\ {\rm mod}\ 5)$.

$$
\begin{array}{l}
((0 , 0),\ (1 , 1),\ (2 , 4),\ (3 , 4),\ (4 , 1)),\ \ ((4 , 0),\ (3 , 1),\ (2 , 3),\ (1 , 1),\ (0 , 0))\\
((0 , 1),\ (1 , 0),\ (2 , 1),\ (3 , 4),\ (4 , 4)),\ \ ((4 , 3),\ (3 ,1),\ (2 , 0),\ (1 , 0),\ (0 , 1))\\
((0 , 4),\ (1 , 1),\ (2 , 0),\ (3 , 1),\ (4 , 4)),\ \ ((4 , 3),\ (3 , 3),\ (2 , 4),\ (1 , 1),\ (0 , 4))\\
((0 , 4),\ (1 , 4),\ (2 , 1),\ (3 , 0),\ (4 , 1)),\ \ ((4 , 0),\ (3 , 2),\ (2 , 0),\ (1 , 4),\ (0 , 4))\\
((0 , 1),\ (1 , 4),\ (2 , 4),\ (3 , 1),\ (4 , 0)),\ \ ((4 , 4),\ (3,
3),\ (2 , 3),\ (1 , 4),\ (0 , 1)).
\end{array}
$$

This super-simple DGDD has 25 disjoint directed trades of volume 2
so has $f\geq\frac{1}{2}$.

A super-simple DGDD of type $(5^7)$  can be constructed on the set $Z_{35}$ with groups
$\{i,i+7,i+14,i+21,i+28\};\ i=0,1,\dots,6$ by developing
the following base blocks with the automorphism $X\mapsto X+7$ or simply by $(+7\ {\rm mod}\ 35)$.

$$
\begin{array}{lllll}
(21, 17, 5, 6, 15),& (5, 17, 4, 7, 34)& & & \\
(15, 33, 20, 23, 21),& (2, 20, 33, 17, 28)&&& \\
(5, 16, 25, 21, 8),& (21, 25, 9, 10, 19)&&&\\
(30, 33, 15, 6, 11),&(31, 8, 25, 33, 30)&&&\\
(19, 11, 8, 24, 28),& (3, 9, 8, 11, 21),& (18, 20, 24, 8, 9)&&\\
(5, 23, 1, 20, 31),& (21, 32, 30, 20, 1),& (17, 20, 30, 12, 18)&&\\
(29, 6, 31, 28, 23),& (9, 13, 4, 28, 31),& (7, 32, 13, 9, 33),& (0, 17, 13, 32, 8),& (8, 32, 31, 6, 12)\\
(28, 29, 5, 9, 3),&&&&\\
(6, 8, 0, 18, 5)&&&&.
\end{array}
$$
This super-simple DGDD has 20  disjoint directed trades of volume 2 in
the first four rows above, the fifth row has five disjoint cyclical trades of
volume 3, the sixth row is a cyclical trade of
volume 15 and the seventh row has five disjoint cyclical trades of
volume 5. So for this super-simple DGDD we have
$f\geq\frac{20+10+8+15}{105}=\frac{53}{105}>\frac{1}{2}$.

A super-simple DGDD of type $(5^9)$ with $f\geq\frac{1}{2}$ can be constructed on the set
$Z_{45}$ with groups $\{i,i+9,i+18,i+27,i+36\};\ i=0,1,\dots,8$ by developing the following base blocks $(+9\ {\rm mod}\ 45)$

$$
\begin{array}{l}
(0,\ 10,\ 11,\ 39,\ 40),\ \ (10,\ 0,\ 2,\ 25,\ 23)\\
(1,\ 11,\ 9,\ 41,\ 43),\ \ \ (11,\ 1,\ 0,\ 24,\ 26)\\
(2,\ 9,\ 10,\ 44,\ 42),\ \ \ (5,\ 39,\ 44,\ 10,\ 16)\\
(3,\ 14,\ 17,\ 40,\ 36),\ \ (21,\ 17,\ 14,\ 27,\ 31)\\
(4,\ 16,\ 15,\ 36,\ 39),\ \ (31,\ 24,\ 25,\ 39,\ 36)\\
(5,\ 17,\ 12,\ 43,\ 37),\ \ (17,\ 5,\ 3,\ 20,\ 24)\\
(6,\ 13,\ 16,\ 38,\ 44),\ \ (24,\ 16,\ 13,\ 35,\ 29)\\
(7,\ 15,\ 13,\ 37,\ 41),\ \ (0,\ 38,\ 37,\ 13\ 12)\\
(8,\ 12,\ 14,\ 42,\ 38),\ \ (7,\ 40,\ 42,\ 14,\ 10)\\
(0,\ 14,\ 15,\ 43,\ 44),\ \ (14,\ 0,\ 6,\ 22,\ 20)\\
(1,\ 17,\ 13,\ 42,\ 39),\ \ (17,\ 1,\ 4,\ 25,\ 18)\\
(2,\ 12,\ 16,\ 40,\ 41),\ \ (12,\ 2,\ 7,\ 24,\ 19)\\
(3,\ 16,\ 11,\ 37,\ 42),\ \ (16,\ 3,\ 2,\ 18,\ 26)\\
(4,\ 10,\ 17,\ 41,\ 38),\ \ (10,\ 4,\ 8,\ 21,\ 24)\\
(5,\ 15,\ 9,\ 38,\ 40),\ \ \ (15,\ 5,\ 0,\ 19,\ 21)\\
(6,\ 9,\ 14,\ 39,\ 37),\ \ \ (9,\ 6,\ 5,\ 26,\ 25)\\
(7,\ 11,\ 12,\ 44,\ 36),\ \ (11,\ 7,\ 3,\ 23,\ 22)\\
(8,\ 13,\ 10,\ 36,\ 43),\ \ (13,\ 8,\ 1,\ 20,\ 23).\\
\end{array}
$$

A super-simple DGDD of type $(6^5)$ with $f\geq\frac{1}{2}$ can be
constructed on the set $V=\{0,1,2,3,4\} \times Z_6$ with groups
$\{(i , 0),\ (i , 1),\ (i , 2),\ (i , 3),\ (i , 4),\ (i , 5)\};\
i=0,1,2,3,4$ by developing the second coordinate of the
following 12 base blocks modulo 6 or simply by $(- , +1\ {\rm mod}\
6)$.
$$
\begin{array}{l}
((4 , 0),\ (0 , 1),\ (1 , 3),\ (3 , 4),\ (2 , 2)),\ \ ((0 , 5),\ (2 , 2),\ (3 , 4),\ (1 , 0),\ (4 , 3))\\
((0 , 3),\ (3 , 4),\ (2 , 5),\ (4 , 2),\ (1 , 1)),\ \ ((2 , 5),\ (3 , 4),\ (1 , 4),\ (0 , 2),\ (4 , 5))\\
((1 , 5),\ (0 , 5),\ (2 , 3),\ (4 , 1),\ (3 , 1)),\ \ ((0 , 5),\ (1 , 5),\ (3 , 5),\ (4 , 5),\ (2 , 5))\\
((4 , 4),\ (2 , 1),\ (1 , 4),\ (0 , 3),\ (3 , 1)),\ \ ((1 , 4),\ (2 , 1),\ (4 , 3),\ (3 , 2),\ (0 , 1))\\
((4 , 3),\ (3 , 5),\ (0 , 5),\ (2 , 4),\ (1 , 4)),\ \ ((3 , 2),\ (4 , 5),\ (1 , 3),\ (2 , 4),\ (0 , 5))\\
((1 , 3),\ (4 , 1),\ (3 , 2),\ (2 , 5),\ (0 , 4)),\ \ ((2 , 2),\ (3 , 5),\ (0 , 0),\ (4 , 1),\ (1 , 3)).\\
\end{array}
$$

A super-simple DGDD of type $(6^6)$ can be constructed on the set
$Z_{36}$ with groups $\{i,i+6,i+12\dots,i+30\};\ i=0,1,\dots,5$
by developing the following base blocks $(+1\ {\rm mod}\ 36)$.
$$(3,\ 1,\ 6,\ 2,\ 16),\ \ (18,\ 1,\ 3,\ 10,\ 26),\ \ (19,\ 10,\ 3,\ 14,\ 0).$$
This super-simple DGDD has  a cyclical trade of volume 108.
So for this super-simple DGDD, we have
$f\geq\frac{54}{108}\geq\frac{1}{2}$.

A super-simple DGDD of type $(10^5)$ can be constructed on the set
$Z_{40}\cup \{\infty_0,\infty_1,\dots,\infty_9\}$ with groups
$\{\infty_0,\infty_1,\dots,\infty_9\}\cup \{i,i+4,i+8\dots,i+36\};\
i=0,1,2,3$ by developing the following base blocks.

$$
\begin{array}{l}
(10,\ 13,\ 28,\ 7,\ \infty_9),\ \ (2,\ 7,\ 28,\ 9,\ \infty_0)\\
(\infty_0,\ 0,\ 5,\ 26,\ 7),\ \ \ \ \ (\infty_9,\ 8,\ 11,\ 26,\ 5)\\
(\infty_1,\ 11,\ 2,\ 1,\ 0),\ \ \ \ \ (2,\ 11,\ 4,\ 33,\ \infty_3) \\
(\infty_3,\ 0,\ 9,\ 2,\ 31),\ \ \ \ \ (9,\ 0,\ 39,\ 38,\ \infty_1) \\
(\infty_4,\ 0,\ 35,\ 14,\ 1),\ \ \ \ (8,\ 31,\ 1,\ 14,\ \infty_5) \\
(\infty_5,\ 10,\ 33,\ 3,\ 16),\ \ (2,\ 37,\ 16,\ 3,\ \infty_4) \\
(\infty_6,\ 35,\ 0,\ 13,\ 30),\ \ (15,\ 13,\ 0,\ 22,\ \infty_8) \\
(\infty_8,\ 17,\ 15,\ 2,\,24),\ \ \ (37,\ 2,\ 15,\ 32,\ \infty_6) \\
(3,\ 6,\ \infty_2,\ 0,\ 17)\\
(39,\ 0,\ \infty_7,\ 10,\ 25).\\
\end{array}
$$
The base blocks in the first eight rows above must be developed by $(+4\ {\rm mod}\ 40)$ and any rows contains 10 disjoint directed trades of volume 2. The last two base blocks must be developed by $(+2\ {\rm mod}\ 40)$ and each of them contains 10 disjoint directed trades of volume 2. So this design has $f\geq\frac{1}{2}$.

A super-simple DGDD of type $(4^5)$ with $f\geq\frac{1}{2}$ can be constructed on the set $Z_{20}$, with groups $\{i,i+1,i+2,i+3\}\ ({\rm mod}\ 20);\ i=1,5,9,13,17  $ and
with the following blocks.

$$
\begin{array}{l}
(1,\ 5,\ 9,\ 13,\ 17)  \ \ \ \ \ (2,\ 6,\ 10,\ 14,\ 17) \ \ \ \ (3,\ 7,\ 11,\ 15,\ 17) \ \ \ \ (4,\ 8,\ 12,\ 16,\ 17) \\
(17,\ 13,\ 12,\ 7,\ 2)  \ \ \ \ (17,\ 14,\ 11,\ 8,\ 1) \ \ \ (17,\ 15,\ 10,\ 5,\ 4) \ \ \ \ \ (17,\ 16,\ 9,\ 6,\ 3)  \\
                                                                         \\
(2,\ 7,\ 9,\ 16,\ 18)  \ \ \ \ \ (3,\ 6,\ 12,\ 13,\ 18) \ \ \ \ (4,\ 5,\ 11,\ 14,\ 18) \ \ \ \ (1,\ 7,\ 12,\ 14,\ 19) \\
(18,\ 16,\ 12,\ 5,\ 1) \ \ \ \ (18,\ 13,\ 9,\ 8,\ 4) \ \ \ \ \ (18,\ 14,\ 10,\ 7,\ 3) \ \ \ \ (19,\ 14,\ 9,\ 5,\ 2)  \\
                                                                         \\
(3,\ 5,\ 10,\ 16,\ 19)  \ \ \ \ (4,\ 6,\ 9,\ 15,\ 19)\ \ \ \ \ (1,\ 6,\ 11,\ 16,\ 0) \ \ \ \ \ (2,\ 5,\ 12,\ 15,\ 0)  \\
(19,\ 16,\ 11\ 7,\ 4)  \ \ \ \ \ (19,\ 15,\ 12,\ 8,\ 3) \ \ \ \ (0,\ 16,\ 10,\ 8,\ 2) \ \ \ \ \ (0,\ 15,\ 9,\ 7,\ 1)   \\
                                                                         \\
(4,\ 7,\ 10,\ 13,\ 0)   \ \ \ \ \ (3,\ 8,\ 9,\ 14,\ 0) \ \ \ \ \ \ \ (2,\ 8,\ 11,\ 13,\ 19)\ \ \ \ (1,\ 8,\ 10,\ 15,\ 18) \\
(0,\ 13,\ 11,\ 5,\ 3)   \ \ \ \ \ (0,\ 14,\ 12,\ 6,\ 4) \ \ \ \ \  (19,\ 13,\ 10,\ 6,\ 1) \ \ \ \ (18,\ 15,\ 11,\ 6,\ 2).\\
\end{array}
$$

A super-simple DGDD of type $(4^6)$ with $f\geq\frac{1}{2}$ can be
constructed on the set $Z_{24}$ with groups $\{i,i+6,i+12,i+18\};\
i=0,1,\dots,5$ by developing the following base blocks $(+1\
{\rm mod}\ 24)$.

$$
\begin{array}{l}
(17,\ 8,\ 0,\ 13,\ 3),\ \ (20,\ 13,\ 0,\ 22,\ 21).\\
\end{array}
$$

A super-simple DGDD of type $(4^{10})$ with $f\geq\frac{1}{2}$ can
be constructed on the set $Z_{36}\cup
\{\infty_0,\infty_1,\infty_2,\\ \infty_3\}$ with groups
$\{\infty_0,\infty_1,\infty_2,\infty_3\}\cup \{i,i+9,i+18,i+27\};\
i=0,1,\dots,8$ by developing the following base blocks $(+2\
{\rm mod}\ 36)$.

$$
\begin{array}{l}
(12,\ 0,\ 7,\ 32,\ 22),\ \ \ \ (7,\ 0,\ \infty_0,\ 11,\ 8)\\
(11,\ 32,\ 1,\ 21,\ 33),\ \ (29,\ 24,\ 23,\ 21,\ 1)\\
(0,\ 29,\ 28,\ 4,\ 6),\ \ \ \ \ (29,\ 0,\ \infty_1,\ 34,\ 17)\\
(13,\ 0,\ \infty_2,\ 30,\ 15)\\
(33,\ 0,\ \infty_3,\ 16,\ 3).\\
\end{array}
$$

A super-simple DGDD of type $(4^{11})$ with $f\geq\frac{1}{2}$ can
be constructed on the set $Z_{44}$ with groups
$\{i,i+11,i+22,i+33\};\ i=0,1,\dots,10$ by developing the
following base blocks $(+1\ {\rm mod}\ 44)$.
$$
\begin{array}{l}
(30,\ 0,\ 36,\ 21,\ 26),\ \ (35,\ 36,\ 0,\ 38,\ 10) \\
(0,\ 24,\ 39,\ 12,\ 37),\ \ (1,\ 24,\ 0,\ 31,\ 28).\\
\end{array}
$$

A super-simple DGDD of type $(4^{16})$ with $f\geq\frac{1}{2}$ can
be constructed on the set $Z_{64}$ with groups
$\{i,i+16,i+32,i+48\};\ i=0,1,\dots,15$ by developing the
following base blocks $(+1\ {\rm mod}\ 64)$.
$$
\begin{array}{l}
(0,\ 1,\ 40,\ 3,\ 47),\ \ \ \ (1,\ 0,\ 25,\ 62,\ 18) \\
(0,\ 4,\ 26,\ 14,\ 35),\ \ (4,\ 0,\ 42,\ 54,\ 33)\\
(0,\ 5,\ 41,\ 11,\ 56),\ \ (5,\ 0,\ 28,\ 58,\ 13).\\
\end{array}
$$

A super-simple DGDD of type $(2^6)$ with $f\geq\frac{1}{2}$ can be constructed on the set $Z_{12}$, with groups $\{i,i+1\}\ ({\rm mod}\ 12);\ i=1,3,5,7,9,11$ and
with the following blocks.
$$
\begin{array}{l}
(3,\ 8,\ 6,\ 9,\ 2),\ \ \ (8,\ 3,\ 11,\ 1,\ 10) \\
(4,\ 7,\ 5,\ 10,\ 1),\ \ (7,\ 4,\ 0,\ 2,\ 9)\\
(9,\ 1,\ 4,\ 6,\ 11),\ \ (10,\ 2,\ 7,\ 11,\ 6)\\
(2,\ 10,\ 3,\ 5,\ 0),\ \ (1,\ 9,\ 8,\ 0,\ 5)\\
(5,\ 11,\ 9,\ 3,\ 7),\ \ (11,\ 5,\ 2,\ 4,\ 8)\\
(6,\ 0,\ 10,\ 8,\ 4),\ \ (0,\ 6,\ 1,\ 7,\ 3).\\
\end{array}
$$

For super-simple DGDD of type $(10^6)$ let $(G,\mathcal{B})$ be a super-simple DGDD of type $(2^6)$ with $f\geq\frac{1}{2}$ with
element set U. We form a super-simple DGDD of type $(10^6)$ with $f\geq\frac{1}{2}$ on the element set $U\times
Z_{5}$. For each $b\in \mathcal{B}$, say $\{x_1,x_2,\dots,x_5\}$, we form a $TD(5,5)$  on
$b\times Z_5$, such that its groups are $\{x_1\}\times Z_5$,
$\{x_2\}\times Z_5$, $\dots$,  $\{x_5\}\times Z_5$.

A super-simple DGDD of type $(6^5\ 8^1)$ can be constructed on the
set $Z_{30}\cup\{\infty_0,\infty_1,\dots,\infty_7\}$ with groups
$\{\infty_0,\infty_1,\dots,\infty_7\}\cup
\{i,i+5,i+10,i+15,i+20,i+25\};\ i=0,1,2,3,4$ by developing the
following base blocks $(+2\ {\rm mod}\ 30)$. $\infty_0$ must be
replaced with $\infty_j ;\ \ j=1,2$ when adding value $6i+2j;\ \
i=0,1,2,3,4$

$$
\begin{array}{l}
(6,\ 14,\ \infty_0,\ 2,\ 0),\ \ \ \ \ (1,\ 23,\ \infty_0,\ 5,\ 7),\ \ (4,\ 3,\ \infty_0,\ 27,\ 16)\\
(7,\ 0,\ \infty_3,\ 14,\ 3),\ \ \ \ \ (1,\ 0,\ \infty_4,\ 2,\ 9),\ \ \ (3,\ 24,\ \infty_5,\ 0,\ 21)\\
(26,\ 27,\ \infty_6,\ 13,\ 0),\ \ (0,\ 13,\ \infty_7\ 11,\ 22).
\end{array}
$$

\noindent This super-simple directed group divisible design has nine disjoint cyclical trades of volume 5 in the first row above and
each of other base blocks is a cyclical trade of volume 15. So for this super-simple DGDD, we have
$f\geq\frac{9*3+5*8}{120}=\frac{67}{120}>\frac{1}{2}$.

A super-simple DGDD of type $(4^8\ 6^1)$ with $f\geq\frac{1}{2}$ can
be constructed on the set $Z_{32}\cup\{\infty_0,\infty_1,\dots,\infty_5\}$ with groups
$\{\infty_0,\infty_1,\dots,\infty_5\}\cup\{i,i+8,i+16,i+24\};\
i=0,1,\dots,7$ by developing the following base blocks $(+2\
{\rm mod}\ 32)$.

$$
\begin{array}{l}
(28,\ 7,\ 0,\ 2,\ 22),\ \ \ \ (25,\ 4,\ \infty_0,\ 0,\ 7)\\
(11,\ 15,\ 0,\ 13,\ 1),\ \ \ (18,\ 3,\ \infty_1,\ 0,\ 15)\\
(29,\ 2,\ \infty_2,\ 0,\ 23)\\
(0,\ 25,\ \infty_3,\ 31,\ 12)\\
(0,\ 9,\ \infty_4,\ 19,\ 18)\\
(31,\ 22,\ \infty_5,\ 0,\ 27).\\
\end{array}
$$

A super-simple DGDD of type $(1^{20}\ 3 ^1)$ with $f\geq\frac{1}{2}$
can be constructed on the set $Z_{20}\cup
\{\infty_0,\infty_1,\infty_2\}$ with groups
$\{\infty_0,\infty_1,\infty_2\}\cup\{i\};\ i=0,1,\dots,19$ by
developing the following base blocks $(+4\ {\rm mod}\ 20)$.

$$
\begin{array}{l}
(2,\ 14,\ 0,\ 18,\ 3),\ \ \ \ \ (19,\ \infty_0,\ 0,\ 14,\ 7),\ (17,\ 7,\ 14,\ 16,\ \infty_2)\\
(2,\ \infty_1,\ 19,\ 1,\ 10),\ \ (4,\ 1,\ 19,\ 3,\ 15)\\
(10,\ 0,\ 5,\ \infty_0,\ 13),\ \ (4,\ 16,\ 5,\ 0,\ 6)\\
(\infty_2,\ 15,\ 14,\ 1,\ 8),\ \ (11,\ 9,\ 1,\ 14,\ 5)\\
(9,\ 15,\ 12,\ \infty_1,\ 0).\\
\end{array}
$$

A super-simple DGDD of type $(1^{16}\ 5 ^1)$ with $f\geq\frac{1}{2}$
can be constructed on the set $Z_{21}$, with groups
$\{16,17,18,19,20\}\cup\{i\};\ i=0,1,\dots,15$ and with the
following blocks.
$$
\begin{array}{l}
(16, 1, 0, 6, 7) \ \ \ \ \ (16, 3, 2, 4, 5) \ \ \ \ \ \ (16, 15, 14, 9, 8) \ \ \ \ (16, 13, 12, 10, 11) \ \ (17, 5, 0, 4, 1)   \\
(19, 15, 10, 6, 0) \ \ (19, 13, 8, 4, 2) \ \ \ \ (17, 11, 10, 14, 15) \ \ (18, 1, 4, 10, 12) \ \ \ \ (18, 0, 5, 11, 13) \\
                                                                                                      \\
(10, 9, 7, 4, 17)\ \ \ (17, 8, 9, 12, 13) \ \ \ (18, 3, 6, 8, 14)\ \ \ \ \ \ (0, 3, 12, 15, 16) \ \ \ (0, 2, 8, 10, 18)  \\
(18, 2, 7, 9, 15) \ \ \ (19, 12, 9, 5, 3) \ \ \ \ (17, 7, 2, 6, 3) \ \ \ \ \ \ \ \ (14, 13, 3, 0, 17) \ \ \ (11, 9, 2, 0, 19)  \\
                                                                                                       \\
(1, 2, 13, 14, 16) \ \ (5, 7, 12, 14, 18) \ \ (4, 7, 11, 8, 16) \ \ \ \ \ \ \ (1, 3, 9, 11, 18) \ \ \ \ (4, 6, 13, 15, 18) \\
(15, 12, 2, 1, 17) \ \ (15, 13, 7, 5, 19) \ \ (19, 14, 11, 7, 1) \ \ \ \ \ (10, 8, 3, 1, 19) \ \ \ \ (14, 12, 6, 4, 19) \\
                                                                                                      \\
(5, 6, 9, 10, 16) \ \ \ (20, 12, 8, 7, 0) \  \ \ (20, 15, 11, 4, 3) \ \ \ \ \ \ (20, 13, 9, 6, 1) \ \ \ \ (20, 14, 10, 5, 2) \\
(8, 11, 6, 5, 17) \ \ \ (6, 2, 11, 12, 20) \ \ (1, 5, 8, 15, 20) \ \ \ \ \ \ \ (3, 7, 10, 13, 20) \ \ (4, 0, 9, 14, 20).  \\
\end{array}
$$

%

\section{Super-simple directed designs with block size 5 and index 1}  

A necessary and sufficient condition for the existence of a
$2$-$(v,5,1)$DD is $v\equiv1,5\ ({\rm mod}\ 10)$, except $v=15$, (see $\cite{DBIBD}$).

In this section  simultaneously we show that the necessary and sufficient condition for the existence of a super-simple $2$-$(v,5,1)$DD is  $v\equiv1,5\ ({\rm mod}\ 10)$ except $v=5,15$ and for these parameters, except possibly $v=11,91$, there exists
a super-simple $(v,5,1)$DD with $f\geq\frac{1}{2}$.

Our principle tool is to apply Wilson's Fundamental construction. For example we have the following
Lemmas.
\begin{lemma}
If there are a $\{K\}$-GDD of type ${g_1}^{u_1}{g_2}^{u_2}\dots {g_N}^{u_N}$, super-simple $2$-$(\alpha g_i+1,5,1)$DD for each
$i,\ i=1,2,\dots,N$ and super-simple DGDDs of type $(\alpha ^k)$ for each $k\in K$ then there exists a super-simple $2$-$(\alpha\sum_{i=1}^{N}g_iu_i+1,5,1)$DD.
\end{lemma}
\begin{proof}
{Let $(G,\mathcal{B})$ be a group divisible design of type ${g_1}^{u_1}{g_2}^{u_2}\dots {g_N}^{u_N}$ with
element set U, blocks of size $k\in K$. We form a super-simple $2$-$(\alpha\sum_{i=1}^{N}g_iu_i+1,5,1)$DD on the element set $U\times
Z_{\alpha} \bigcup\{\infty\}$. For each $b\in \mathcal{B}$ of size $
k$, say $\{x_1,x_2,\dots,x_k\}$, we form a super-simple $DGDD$ of type $(\alpha ^k)$  on
$b\times Z_{\alpha}$, such that its groups are $\{x_1\}\times Z_{\alpha}$,
$\{x_2\}\times Z_{\alpha}$, $\dots$,  $\{x_k\}\times Z_{\alpha}$. For
each group $g\in G$ of size $g_i,\ i=1,2,\dots,N$ we substitute a super-simple
$2$-$(\alpha g_i+1,5,1)$DD.
}
\end{proof}
In particular, if there is a constant c such that each of the super-simple DGDDs of type $(\alpha^k)$ and each of the super-simple
$2$-$(\alpha g_i+1,5,1)$DDs has $f\geq c$, then the resulting super-simple $2$-$(v,5,1)$DD in Lemma 3.1 also has $f\geq c$.

\begin{lemma}
If there are a super-simple DGDD of type $({g_1}^{u_1}{g_2}^{u_2}\dots {g_N}^{u_N})$ with block size 5 and index 1, super-simple $2$-$(\alpha g_i+1,5,1)$DDs for each
$i,\ i=1,2,\dots,N$ and $5$-GDDs of type $\alpha ^5$ then there exists a super-simple $2$-$(\alpha\sum_{i=1}^{N}g_iu_i+1,5,1)$DD.
\end{lemma}
In particular, if there is a constant c such that the super-simple DGDD of type $({g_1}^{u_1}{g_2}^{u_2}\dots {g_N}^{u_N})$ and each of the super-simple
$2$-$(\alpha g_i+1,5,1)$DDs has $f\geq c$, then the resulting super-simple $2$-$(v,5,1)$DD in Lemma 3.2 also has $f\geq c$.

\begin{lemma}
For all $n\geq5$ except $n\in\{6,10,14,18,22\}$ there exists a $\{5,6\}$-GDD of type $n^5x^1$, where $x<n$.
\end{lemma}
\begin{proof}
{For all $n\geq5$ except $n\in\{6,10,14,18,22\}$ there exists a TD$(6,n)$ $\cite{PBD}$. We can remove $y$ points from one group to obtain a $\{5,6\}$-GDD of type $n^5x^1$, where $x+y=n$. }
\end{proof}

\begin{lemma}
For all $n$ except $n\in\{2,11,17,23,32\}$ there exists a $\{5,6\}$-GDD of type $5^{4n+1}x^{1}$, where $x\leq 5n$.
\end{lemma}
\begin{proof}
{For all $n$ except $n\in \{2,11,17,23,32\}$ there exists a RBIBD$(20n+5,5,1)$ $\cite{PBD}$. This design has $5n+1$ parallel classes so for all n except $n\in \{2,11,17,23,32\}$ exists a RGDD of type $5^{4n+1}$. We can add a group of size $x$, where $x\leq 5n$ to obtain a $\{5,6\}$-GDD of type $5^{4n+1}x^{1}$. }
\end{proof}

\begin{lemma}
If $q$ is a prime power, then for all $k\leq q+1$ there exists a TD$(k,q)$.
\end{lemma}
\begin{proof}
{It is well-known that a TD$(k,g)$ is equivalent to $k-2$ mutually
orthogonal Latin square (MOLS) of order $g$ and also for all $q$
there exists $q-1$ MOLS of order $q$. So for all $q$ there exists a
TD$(q+1,q)$. Now by deleting all the elements of $x$ groups, a
TD$(k,q)$ can be obtained, where $x+k=q+1$. }
\end{proof}

\Large{$v\equiv1\ ({\rm mod}\ 10)$}

\normalsize
\begin{lemma}
For $v=11$ there exists a super-simple $2$-$(11,5,1)$DD with $f\geq\frac{2}{11}$ and for all $v\in\{21,31,41\}$ there exists a super-simple $2$-$(v,5,1)$DD with $f\geq\frac{1}{2}$.
\end{lemma}
\begin{proof}
{For $v=11$, the blocks
$\{$$\bf{12345,51678},$$\normalsize6319a,98410,270a1,30286,85a29,74962,4a837,\\09753,a6504\}$
construct a  super-simple $2$-$(11,5,1)$DD in which the two 5-tuples
in bold font form its defining set. So for this super-simple
$2$-$(11,5,1)$DD we have $f\geq\frac{2}{11}$.

For $v=21$, develop the two base blocks $$(0,\ 1,\ 4,\ 14,\ 16),\ \ \
\ (1,\ 0,\ 18,\ 8,\ 6)$$ modulo 21 to obtain a super-simple
$2$-$(21,5,1)$DD with $f\geq\frac{1}{2}$, (see
$\cite{pentagon2004}$).

For $v=31$, using the following base blocks $(+1\ {\rm mod}\
31)$, a super-simple $2$-$(31,5,1)$DD with $f\geq\frac{1}{2}$ can be
constructed, (see $\cite{pentagon2004}$).
$$
\begin{array}{l}
(0,\ 4,\ 29,\ 12,\ 28),\ \ (13,\ 4,\ 0,\ 5,\ 15),\ \ (27,\ 2,\ 15,\ 5,\ 22).\\
\end{array}
$$
These base blocks contain a cyclical trade of volume 93. So this design has $f>\frac{1}{2}$.

For $v=41$, using the following base blocks $(+1\ {\rm mod}\
41)$, a super-simple $2$-$(41,5,1)$DD with $f\geq\frac{1}{2}$ can be
constructed, (see $\cite{pentagon2004}$).
$$
\begin{array}{l}
(0,\ 1,\ 4,\ 11,\ 29),\ \ (1,\ 0,\ 38,\ 31,\ 13) \\
(0,\ 2,\ 8,\ 17,\ 22),\ \ (2,\ 0,\ 35,\ 26,\ 21).\\
\end{array}
$$
}
\end{proof}
\begin{theorem}
For all $v\equiv1\ ({\rm mod}\ 10);\ v>11$, except possibly
$v=91$, there exists a super-simple $2$-$(v,5,1)$DD with
$f\geq\frac{1}{2}$.
\end{theorem}
\begin{proof}
{For $v=21,\ 31,\ 41$, see the previous Lemma. For all $n\equiv1\
({\rm mod}\ 2)$ except $n\in A =
\{11,\dots,19,23,27,\dots,33,39,43,51,59,71,75,83,87,95,99,107,111,113,115,119,139,179\}$
there exists a PBD$(n, \{5, 7, 9\})$ $\cite{mullin}$. We can remove
one point from this PBD to form a $\{5, 7, 9\}$-GDD of order $n-1$
of type $4^{a}6^{b}8^{c}$, where $a,\ b,\ c$ are non-negative
integers. Applying Lemma 3.1 with $\alpha=5$ and using the required
super-simple DGDDs and also the required super-simple
$2$-$(5g_i+1,5,1)$DDs, $g_i\in\{4,6,8\}$ with $f\geq\frac{1}{2}$, we
obtain a super-simple $2$-$(v,5,1)$DD with
 $f\geq\frac{1}{2}$.

Now for the rest of the values $n\in A$, we
construct a super-simple $2$-$(5(n-1)+1,5,1)$DD with $f\geq\frac{1}{2}$ as follows.\\
\noindent $\bullet\ \ \ n=11$: using the following base blocks
$(+1\ {\rm mod}\ 51)$, a super-simple $2$-$(51,5,1)$DD can be
constructed, (see $\cite{pentagon2004}$).
$$
\begin{array}{l}
(1,\ 0,\ 14,\ 31,\ 35),\ \ \ (14,\ 0,\ 42,\ 24,\ 43) \\
(42,\ 0,\ 11,\ 18,\ 16),\ \ (11,\ 0,\ 6,\ 47,\ 8),\ \ (21,\ 44,\ 15,\ 8,\ 47).\\
\end{array}
$$
This super-simple $2$-$(51,5,1)$DD has 51 disjoint directed trades of volume 2 in the first row above and three cyclical directed trades of volume 51 in the other row. So for this design we have $f\geq\frac{129}{255}>\frac{1}{2}$.\\
\noindent $\bullet\ \ \ n=13$: using the following base blocks
 $(+1\ {\rm mod}\ 61)$, a super-simple $2$-$(61,5,1)$DD
with $f\geq\frac{1}{2}$ can be constructed, (see
$\cite{pentagon2004}$).
$$
\begin{array}{l}
(0,\ 1,\ 3,\ 13,\ 34),\ \ \ (1,\ 0,\ 59,\ 49,\ 28) \\
(0,\ 4,\ 9,\ 23,\ 45),\ \ \ (4,\ 0,\ 56,\ 42,\ 20) \\
(0,\ 6,\ 17,\ 24,\ 32),\ \ (6,\ 0,\ 50,\ 43,\ 35). \\
\end{array}
$$
\noindent $\bullet\ \ \ n=15$: first we form a super-simple  DGDD of type $(10^7)$
with the following base blocks, (see $\cite{pentagon}$).
$$
\begin{array}{l}
(39,\ 2,\ 28,\ 0,\ 1),\ \ \ \ \ \ \ (46,\ 53,\ 43,\ 2,\ 39),\ \ (\infty_5,\ 33,\ 28,\ 2,\ 43) \\
(47,\ 22,\ \infty_6,\ 0,\ 39) \\
(15,\ 35,\ \infty_0,\ 0,\ 43),\ \ (31,\ 35,\ 15,\ 2,\ \infty_1) \\
(0,\ 51,\ 2,\ 5,\ \infty_8),\ \ \ \ \ (\infty_8,\ 5,\ 56,\ 7,\ 10)\\
(\infty_9,\ 7,\ 58,\ 9,\ 12),\ \ \ (2,\ 53,\ 4,\ 7,\ \infty_9). \\
\end{array}
$$
The pentagons in the first three rows above must be developed by $(+1\ {\rm mod}\ 60)$, $\infty_6$ must be replaced by $\infty_7$ when adding odd values, $\infty_0$ must be replaced by $\infty_y$ when adding any values congruent to $y\ ({\rm mod}\ 6)$ and for $x= 1,\ 5$, $\infty_x$ must be replaced by $\infty_{\alpha};\ \alpha\equiv x+y\ ({\rm mod}\ 6)$ when adding any value $y$.
The pentagons in the last two rows above must be developed by $(+4\ {\rm mod}\ 60)$. Eventually to make a super-simple $2$-$(71,5,1)$DD we can add a point to all groups in the super-simple  DGDD of type $(10^7)$ and put a $2$-$(11,5,1)$DD.
This super-simple $2$-$(71,5,1)$DD has 60 cyclical directed trades of volume 3 in the first row and each of other rows above contain 30, 60, 15, 15 disjoint directed trades of volume 2, respectively.  So this design has $f\geq\frac{2*7+2*60+30+60+15+15}{497}>\frac{1}{2}$.\\
\noindent $\bullet\ \ \ n=17$: using the following base blocks
$(+1\ {\rm mod}\ 81)$, a super-simple $2$-$(81,5,1)$DD with
$f\geq\frac{1}{2}$ can be constructed, (see $\cite{pentagon2004}$).
$$
\begin{array}{l}
(0,\ 1,\ 5,\ 12,\ 26),\ \ \ (1,\ 0,\ 77,\ 70,\ 56) \\
(0,\ 2,\ 10,\ 40,\ 64),\ \ (2,\ 0,\ 73,\ 43,\ 19) \\
(0,\ 3,\ 18,\ 47,\ 53),\ \ (3,\ 0,\ 66,\ 37,\ 31) \\
(0,\ 9,\ 32,\ 48,\ 68),\ \ (9,\ 0,\ 58,\ 42,\ 22).
\end{array}
$$
\noindent $\bullet\ \ \ n=23$: using the following base blocks $(+1\ {\rm mod}\ 3,\ +1\ {\rm mod}\ 37)$,
a super-simple $2$-$(111,5,1)$DD with $f\geq\frac{1}{2}$ can be constructed, (see $\cite{DBIBD}$).
$$
\begin{array}{l}
((0 , 36),\ (1 , 0),\ (1 , 35),\ (2 , 28),\ (2 , 7)),\ \ ((2 , 36),\ (1 , 35),\ (1 , 0),\ (0 , 22),\ (0 , 13)),\ \\
((0 , 15),\ (0 , 22),\ (1 , 0),\ (0 , 33),\ (0 , 4)),\ \ ((2 , 6),\ (0 , 16),\ (0 , 33),\ (1 , 0),\ (1 , 12)),\ \\
((0 , 16),\ (2 , 6),\ (2 , 26),\ (1 , 24),\ (1 , 8))\\
((0 , 0),\  (1 , 26),\ (1 , 11),\ (2 , 14),\ (2 , 23)),\ ((2 , 0),\ (1 , 11),\ (1 , 26),\ (0 , 6),\ (0 , 31))\\
((0 , 20),\ (0 , 17),\ (1 , 0),\ (0 , 30),\ (0 , 7)),\ \ ((1 , 8),\ (0 , 24),\ (0 , 17),\ (0 , 20),\ (2 , 8))\\
((0 , 35),\ (0 , 2),\ (1 , 0),\ (0 , 3),\ (0 , 34)),\ \ \ ((2 , 16),\ (0 , 11),\ (0 , 34),\ (0 , 3),\ (1 , 16)).\\
\end{array}
$$
This super-simple $2$-$(111,5,1)$DD has 333 disjoint directed trades
of volume 2 in the last three rows above and the first three rows
contain a cyclical trade of volume 555. So this design has
$f\geq\frac{611}{1221}>\frac{1}{2}$.

\noindent $\bullet\ \ \ n=27$: using the following base blocks $(+1\ {\rm mod}\ 131)$, a super-simple $2$-$(131,5,1)$DD with $f\geq\frac{1}{2}$ can be constructed, (see $\cite{pentagon2004}$).
$$
\begin{array}{l}
(75,\ 4,\ 1,\ 127,\ 32),\ \ \ \ \ \ \ (84,\ 89,\ 56,\ 1,\ 4),\ \ \ \ \ \ \ \ \ (1,\ 56,\ 5,\ 86,\ 38),\ \ \ \ \ \ \ (86,\ 5,\ 49,\ 23,\ 19) \\
(87,\ 113,\ 73,\ 49,\ 5),\ \ \ \ \ \ (49,\ 73,\ 24,\ 26,\ 64),\ \ \ \ \ \ (26,\ 24,\ 11,\ 33,\ 58),\ \ \ (65,\ 31,\ 43,\ 11,\ 24) \\
(101,\ 121,\ 120,\ 43,\ 31),\ \ (104,\ 51,\ 62,\ 120,\ 121),\ \ (6,\ 78,\ 68,\ 62,\ 51),\ \ \ (62,\ 68,\ 23,\ 89,\ 33) \\
(89,\ 23,\ 118,\ 58,\ 97).\\
\end{array}
$$
These base blocks contain a cyclical trade of volume 1703. So this design has $f>\frac{1}{2}$.

\noindent $\bullet\ \ \ n=29$: begin with a super-simple DGDD of type $(5^7)$. Using Lemma 3.2 with $\alpha=4$, a $5$-GDD of type $4^5$ and a
super-simple $2$-$(21,5,1)$DD we construct a super-simple $2$-$(141,5,1)$DD with $f\geq\frac{1}{2}$.\\
\noindent $\bullet\ \ \ n=31$: begin with a $5$-GDD of type $5^5$. Using Lemma 3.1 with $\alpha=6$,  we obtain a super-simple $2$-$(151,5,1)$DD with $f\geq\frac{1}{2}$.\\
\noindent $\bullet\ \ \ n=33$: using the following base blocks $(+1\ {\rm mod}\ 160)$, a super-simple DGDD of type $(20^8)$ with $f\geq\frac{1}{2}$ can be constructed, (see $\cite{HSPS}$).
$$
\begin{array}{l}
(0,\ 11,\ 21,\ 38,\ 135),\ \ (11,\ 0,\ 150,\ 133,\ 36)\\
(0,\ 3,\ 33,\ 95,\ 110),\ \ \ (3,\ 0,\ 130,\ 68,\ 53)\\
(0,\ 2,\ 31,\ 89,\ 101),\ \ \ (2,\ 0,\ 131,\ 73,\ 61)\\
(0,\ 6,\ 55,\ 81,\ 115),\ \ \ (6,\ 0,\ 111,\ 85,\ 51)\\
(0,\ 7,\ 35,\ 76,\ 113),\ \ \ (7,\ 0,\ 132,\ 91,\ 54)\\
(0,\ 1,\ 23,\ 67,\ 141),\ \ \ (1,\ 0,\ 138,\ 94,\ 20)\\
(0,\ 9,\ 14,\ 117,\ 156),\ \ (9,\ 0,\ 155,\ 52,\ 13).\\
\end{array}
$$
Now, using an extra point, fill in the holes and apply the super-simple $2$-$(21,5,1)$DD to construct a super-simple $2$-$(161,5,1)$DD with $f\geq\frac{1}{2}$.\\
\noindent $\bullet\ \ \ n=39$: begin with a $6$-GDD of type $5^6$ $\cite{PBD}$. Give weight 6 to all points in the five groups and weight 8 to the last group then using the super-simple DGDD of type $(6^{5} 8^{1})$, the super-simple $2$-$(31,5,1)$DD and the super-simple $2$-$(41,5,1)$DD we obtain a super-simple $2$-$(191,5,1)$DD with $f\geq\frac{1}{2}$.\\
\noindent $\bullet\ \ \ n=43$, there exists a TD$(6,9)$ $\cite{PBD}$, add a point $y$ to the groups, then delete a different point so as to form a $\{6,10\}$-GDD of type $5^{9} 9$. Give weight 4 to all points in the groups of size 5 and in the last group give weight 4 to $y$, weight 4 to 6 points and weight 0 to the rest. The point $y$ is in the blocks of size 10 so we replace any of them by a super-simple DGDD of type $(4^{10})$ and for blocks of size 6 use a super simple DGDD of type $(4^5)$, $(4^6)$. This yields a super-simple DGDD of type $(20^9\ 28)$, finally add three new points to groups, then on each of 9 groups of size 20 together with the three new points construct a super-simple DGDD of type $(1^{20}\ 3^1)$ and on the last group together with the three new points construct a super-simple $2$-$(31,5,1)$DD.\\
\noindent $\bullet\ \ \ n=51$: begin with a $5$-GDD of type $5^5$ $\cite{PBD}$. Using Lemma 3.1 with $\alpha=10$,  we obtain a super-simple $2$-$(251,5,1)$DD with $f\geq\frac{1}{2}$.\\
\noindent $\bullet\ \ \ n=59$: in Lemma 3.4 it is shown that there exists a $\{5,6\}$-GDD of type $5^{5}4^{1}$. Now using Lemma 3.1 with $\alpha=10$ we get a super-simple $2$-$(291,5,1)$DD with $f\geq\frac{1}{2}$.\\
\noindent $\bullet\ \ \ n=71$: begin with a $5$-GDD of type $14^5$ $\cite{PBD}$. Using Lemma 3.1 with $\alpha=5$,  we obtain a super-simple $2$-$(351,5,1)$DD with $f\geq\frac{1}{2}$.\\
\noindent $\bullet\ \ \ n\in \{75,83\}$: in Lemma 3.3 it is shown that there exist a $\{5,6\}$-GDD of type $7^{5}2^{1}$ and a $\{5,6\}$-GDD of type $7^{5}6^{1}$. Now using Lemma 3.1 with $\alpha=10$ we get a super-simple $2$-$(371,5,1)$DD and a super-simple $2$-$(411,5,1)$DD with $f\geq\frac{1}{2}$, respectively.\\
\noindent $\bullet\ \ \ n\in \{87,95,99,107,111,113,115,119,139,179\}$: in Lemma 3.4 it is shown that for all $x\in \{13,17,21,25\}$ there exists a $\{5,6\}$-GDD of type $5^{x}k^{1}$, where $k\leq\frac{5(x-1)}{4}$. Give weight 6 to all points in the first $x$ groups and weight 0 or 6 or 8 to the points in the last group then use the super-simple DGDD of type $(6^{5} 8^{1})$, super-simple DGDD of type $(6^{5})$, super-simple DGDD of type $(6^{6})$. If we give weight 6 to all points in the first $x$ groups and weight 6 to two points and weight 8 to one point in the last group or weight 8 to all points in the last group of $\{5,6\}$-GDD of type $5^{x}k^{1}$ then we can form a super-simple DGDD of type $(30^{x} 20^{1})$ or super-simple DGDD of type $(30^{x} (10(x-1))^{1})$, respectively. Generally for all $x\in \{13,17,21,25\}$ there exists a super-simple DGDD of type $(30^{x} g^{1})$ where $20\leq g\leq10(x-1)$. Since there is a super-simple $2$-$(g+1,5,1)$DD where $g+1\equiv1\ ({\rm mod}\ 10)$ (except $g=90$), we can add a new point to this super-simple DGDD of type $(30^{x} g^{1})$ and construct a super-simple $2$-$(30x+g+1,5,1)$DD with $f\geq\frac{1}{2}$. For example for $n=115$ we can take $x=17$ and give weight 6 to all points in the first 17 groups and weight 6 to ten points of the last group of a $\{5,6\}$-GDD of type $5^{17}20^{1}$. This yields a super-simple DGDD of type $(30^{17} 60^{1})$ now add a new point to this super-simple DGDD to construct a super-simple $2$-$(571,5,1)$DD with $f\geq\frac{1}{2}$. }
\end{proof}
\Large{$v\equiv5\ ({\rm mod}\ 20)$}

\normalsize
\begin{lemma}
For all  $v\in\{25,45,65,85,105\}$ there exists a super-simple $2$-$(v,5,1)$DD with $f\geq\frac{1}{2}$.
\end{lemma}
\begin{proof}
{For $v=25$, develop the following base blocks $(-,+1\ {\rm mod}\ 12)$ to obtain a super-simple $2$-$(25,5,1)$DD with $f\geq\frac{1}{2}$.
$$
\begin{array}{l}
((1 , 1),\ (0 , 0),\ (0 , 3),\ (0 , 1),\ (0 , 7)),\ \ \ \ \ \ ((0 , 3),\ (0 , 0),\ (1 , 0),\ (1 , 8),\ (0 , 5))\\
((1 , 9),\ (1 , 11),\ (0 , 0),\ (1 , 3),\ (1 , 10)),\ \ \ ((1 , 0),\ (0 , 8),\ (0 , 4),\ (1 , 10),\ (1 , 3))\\
((0 , 0),\ (1 , 4),\ \infty,\ (1 , 1),\ (0 , 11)).\\
\end{array}
$$
For $v=45$, develop the following base blocks under the groups generated by $(0,\dots,21)\ (22,\dots,43)$ to construct a super-simple $2$-$(45,5,1)$DD with $f\geq\frac{1}{2}$, (see $\cite{pentagon2004}$).
$$
\begin{array}{l}
(0,\ 12,\ 25,\ 4,\ 41),\ \ \ \ (15,\ 37,\ 14,\ 12,\ 0)\\
(35,\ 26,\ 0,\ 31,\ 30),\ \ (24,\ 4,\ 22,\ 30,\ 31)\\
(0,\ 13,\ 16,\ 33,\ 9),\ \ \ \ (9,\ 33,\ 43,\ 23,\ 11)\\
(43,\ 32,\ 3,\ 35,\ 4)\\
(36,\ 33,\ 0,\ 17,\ 6),\ \ \ \ (17,\ 0,\ 23,\ 38,\ \infty)\\
(37,\ 34,\ 1,\ 18,\ 7),\ \ \ \ (\infty,\ 18,\ 1,\ 24,\ 39).\\
\end{array}
$$
The last two rows above is developed by $+2$ under the groups $(0,\dots,21)\ (22,\dots,43)$ and contain 22 disjoint directed trades of volume 2, the fourth row and the first three rows above contain 11 and 66 disjoint directed trades of volume 2, respectively.

For $v=65$, develop the following base blocks under the groups generated by $(0,\dots,31)\ (32,\dots,63)$ to construct
a super-simple $2$-$(65,5,1)$DD with $f\geq\frac{1}{2}$, (see $\cite{pentagon2004}$).
$$
\begin{array}{l}
(23,\ 12,\ 50,\ 0,\ 10),\ \ \ \ (10,\ 0,\ 25,\ 4,\ 51)\\
(29,\ 14,\ 57,\ 45,\ 0),\ \ \ \ (24,\ 16,\ 0,\ 45,\ 46)\\
(1,\ 30,\ 0,\ 34,\ 40),\ \ \ \ \ (48,\ 47,\ 34,\ 0,\ 44)\\
(0,\ 56,\ 35,\ 49,\ 14),\ \ \ \ (33,\ 56,\ 0,\ 37,\ 28)\\
(57,\ 47,\ 32,\ 59,\ 21),\ \ \ (61,\ 38,\ 9,\ 59,\ 32)\\
(56,\ 32,\ 48,\ 2,\ 25)\\
(\infty,\ 50,\ 33,\ 6,\ 1),\ \ \ \ \ \ (26,\ 1,\ 6,\ 58,\ 7)\\
(49,\ 32,\ 5,\ 0,\ \infty),\ \ \ \ \ \ (25,\ 0,\ 5,\ 57,\ 6).\\
\end{array}
$$
The last two rows above is developed by $+2$ under the groups $(0,\dots,31)\ (32,\dots,63)$ and contain 32 disjoint directed trades of volume 2, the sixth row and the first five rows above contain 16 and 160 disjoint directed trades of volume 2, respectively.

For $v=85$, develop the following base blocks $(+2\ {\rm mod}\ 74)$ and construct a $2-(11,5,1)DD$ on points $\infty_0,\dots,\infty_{10}$ to obtain a super-simple $2$-$(85,5,1)$DD with $f\geq\frac{1}{2}$, (see $\cite{pentagon}$).
$$
\begin{array}{l}
(0,\ 10,\ 59,\ 18,\ 61),\ \ \ (28,\ 1,\ 53,\ 10,\ 0)\\
(5,\ 1,\ 59,\ 69,\ 31),\ \ \ \ \ (59,\ 1,\ 63,\ 29,\ 4)\\
(21,\ 0,\ 69,\ 7,\ 29),\ \ \ \ \ (0,\ 21,\ \infty_0,\ 3,\ 72)\\
(0,\ 9,\ \infty_1,\ 41,\ 14),\ \ \ \ (9,\ 0,\ \infty_2,\ 23,\ 20)\\
(47,\ 32,\ \infty_3,\ 0,\ 67),\ \ (67,\ 0,\ \infty_4,\ 65,\ 48)\\
(0,\ 52,\ 36,\ 12,\ 6),\ \ \ \ \ (36,\ 52,\ 16,\ 2,\ 35),\ \ \ \ (11,\ 52,\ 0,\ 4,\ 17) \\
(13,\ 42,\ \infty_5,\ 0,\ 31)\\
(39,\ 0,\ \infty_6,\ 62,\ 15)\\
(0,\ 37,\ \infty_7,\ 1,\ 2)\\
(4,\ 57,\ \infty_8,\ 49,\ 0)\\
(37,\ 8,\ \infty_9,\ 0,\ 5)\\
(31,\ 44,\ \infty_{10},\ 0,\ 55).\\
\end{array}
$$
This super-simple $2$-$(85,5,1)$DD has 185 disjoint directed trades of volume 2 in the first five rows above, the sixth row contains a cyclical trade of volume 111 and any of the last six rows above is a cyclical trade of volume 37. So this design has $f\geq\frac{185+56+6*19+2}{714}\geq\frac{1}{2}$.

To construct a super-simple $2$-$(105,5,1)$DD use a TD$(5,5)$, add a
point $y$ to the groups then delete a different point so as to form
a $\{5,6\}$-GDD of type $4^{5} 5$. Give weight 4 to the points in
all groups. This yields a super-simple DGDD of type $(16^{5} 20)$.
Finally  add five new points to groups, then on each group of size
16 together with the five new points construct a super-simple DGDD
of type $(1^{16}\ 5^1)$ and on the last group together with the five
new points construct a super-simple $2$-$(25,5,1)$DD. This design
has $f\geq\frac{1}{2}$.}
\end{proof}
\begin{theorem}
For all $v\equiv5\ ({\rm mod}\ 20);\ v>5$ there exists a super-simple $2$-$(v,5,1)$DD with $f\geq\frac{1}{2}$.
\end{theorem}
\begin{proof}
{For $v=25,\ 45,\ 65,\ 85$, see the previous Lemma. In Lemma 3.4 it is shown that for all $n$ except $n\in \{2,11,17,23,32\}$ there exists a $\{5,6\}$-GDD of type $5^{4n+1}x^{1}$ where $x\leq 5n;\ x\equiv1\ ({\rm mod}\ 5)$. Give weight 4 to all points in the groups then using the super-simple DGDDs of type $(4^{5})$ and of type $(4^{6})$, the super-simple $2$-$(21,5,1)$DD, the super-simple $2$-$(4x+1,5,1)$DD we obtain a super-simple $2$-$(80n+4x+21,5,1)$DD with $f\geq\frac{1}{2}$. For $n\in \{11,17,23,32\}$ we can use the $\{5,6\}$-GDD of type $5^{4t+1}x^{1}$ that $t\in \{10,16,22,31\}$ and $x\geq 31;\ x\equiv1\ ({\rm mod}\ 5)$. Eventually it remains the values $\{125,\dots,265,325,345,425\}$.\\
\noindent $\bullet\ \ \ v=125$: there exists a super-simple DGDD of type $(5^{5})$. Give weight 5 to the points in all groups of size 5 and use a GDD of type $5^5$. This yields a super-simple DGDD of type $(25^{5})$with $f\geq\frac{1}{2}$. It's sufficient to construct a super-simple $2$-$(25,5,1)$DD with $f\geq\frac{1}{2}$ on the points of each group.\\
\noindent $\bullet\ \ \ v\in\{145,185,265\}$: there exist $\{5,6\}$-GDDs of type $6^6,\ 5^{8} 6$ and $11^6$ $\cite{pentagon2004}$. Use Lemma 3.1 with $\alpha=4$.\\
\noindent $\bullet\ \ \ v=165$: there exists a super-simple DGDD of type $(4^{10})$. Give weight 4 to the points in all groups of size 4 and use a GDD of type $4^5$. This yields a super-simple DGDD of type $(16^{10})$ with $f\geq\frac{1}{2}$. Now add five new points to groups, then on each of 9 of the 10 groups together with the five new points construct a super-simple DGDD of type $(1^{16}\ 5^1)$ and on the last group together with the five new points construct a super-simple $2$-$(21,5,1)$DD.\\
\noindent $\bullet\ \ \ v=205$: there exists a TD$(6,17)$ $\cite{PBD}$. Give weight 2 to the points in all groups  then using the super-simple DGDD of type $(2^6)$ and a super-simple $2$-$(35,5,1)$DD with $f\geq\frac{1}{2}$ (refer to $v\equiv15\ ({\rm mod}\ 20)$) we obtain a super-simple $2$-$(205,5,1)$DD with $f\geq\frac{1}{2}$.\\
\noindent $\bullet\ \ \ v=225$: there exists a super-simple DGDD of type $(5^9)$. Give weight 5 to the points in all groups then replace any groups by a super-simple $2$-$(25,5,1)$DD and any blocks by a GDD of type $5^5$.\\
\noindent $\bullet\ \ \ v=245$: there exists a super-simple DGDD of type $(5^7)$. Give weight 7 to the points in all groups then replace any groups by a super-simple $2$-$(35,5,1)$DD and any blocks by a GDD of type $7^5$.\\
\noindent $\bullet\ \ \ v\in\{325,345,425\}$: in Lemma 3.3 it is shown that  there exist $\{5,6\}$-GDDs of types $15^{5} 6,\ 15^{5} 11$ and $20^{5}6$. Now use Lemma 3.1 with $\alpha=4$. }
\end{proof}
\Large{$v\equiv15\ ({\rm mod}\ 20)$}
\normalsize
\begin{lemma}
For all  $v\in\{35,55,75,95,115,135\}$ there exists a super-simple $2$-$(v,5,1)$DD with $f\geq\frac{1}{2}$.
\end{lemma}
\begin{proof}
{For $v=35$, develop the following base blocks under the groups generated by $(1,\dots,17)\\ (18,\dots,34)$ to construct a super-simple $2$-$(35,5,1)$DD with $f\geq\frac{1}{2}$, (see $\cite{pentagon2004}$).
$$
\begin{array}{l}
(1,\ 9,\ 4,\ 6,\ 2),\ \ \ \ \ \ \ \ (33,\ 27,\ 32,\ 9,\ 1)\\
(6,\ 30,\ 23,\ 10,\ 19),\ \ (9,\ 32,\ 3,\ 23,\ 30)\\
(1,\ 34,\ 26,\ 29,\ 7),\ \ \ \ (28,\ 17,\ 32,\ 7,\ 29)\\
(18,\ 2,\ \infty,\ 1,\ 20).\\
\end{array}
$$

For $v=55$, develop the following base blocks $(+2\ {\rm mod}\ 54)$ to construct a super-simple $2$-$(55,5,1)$DD with $f\geq\frac{1}{2}$, (see $\cite{pentagon}$).
$$
\begin{array}{l}
(0,\ 28,\ 7,\ 16,\ 8),\ \ \ \ \ (3,\ 25,\ 10,\ 7,\ 28)\\
(0,\ 53,\ 27,\ 4,\ 17),\ \ \ \ (11,\ 7,\ 17,\ 4,\ 34)\\
(48,\ 34,\ 39,\ 9,\ 0),\ \ \ \ (37,\ 9,\ 39,\ 25,\ 17)\\
(30,\ 40,\ 0,\ 11,\ 23),\ \ (18,\ 40,\ 30,\ 12,\ 7)\\
(52,\ 17,\ 0,\ 1,\ 50),\ \ \ \ (23,\ 17,\ 52,\ 37,\ 36)\\
(11,\ 22,\ \infty,\ 0,\ 9).\\
\end{array}
$$

For $v=75$, develop the following base blocks $(+2\ {\rm mod}\ 74)$ to construct a super-simple $2$-$(75,5,1)$DD with $f\geq\frac{1}{2}$, (see $\cite{pentagon}$).
$$
\begin{array}{l}
(6,\ 0,\ 47,\ 44,\ 3),\ \ \ \ \ \ (40,\ 26,\ 65,\ 3,\ 44)\\
(20,\ 13,\ 32,\ 0,\ 40),\ \ \ (72,\ 27,\ 32,\ 13,\ 22)\\
(17,\ 42,\ 0,\ 31,\ 48),\ \ \ (37,\ 31,\ 0,\ 19,\ 70)\\
(29,\ 64,\ 63,\ 52,\ 0),\ \ \ (68,\ 53,\ 52,\ 63,\ 44)\\
(26,\ 5,\ 43,\ 69,\ 1),\ \ \ \ (5,\ 26,\ 52,\ 8,\ 54)\\
(25,\ 3,\ 27,\ 1,\ 43),\ \ \ \ (54,\ 16,\ 3,\ 25,\ 32)\\
(60,\ 7,\ 61,\ 0,\ 65),\ \ \ \ (53,\ 54,\ 15,\ 61,\ 7)\\
(2,\ 15,\ \infty\ 59,\ 0).\\
\end{array}
$$

For $v=95$, develop the following base blocks as follows, (see $\cite{pentagon}$).\\
\noindent The base blocks in the first three rows are developed by  $(+6\ {\rm mod}\ 84)$ and each of them forms 7 blocks. In the base blocks $(49,\ 7,\ 6,\ 48,\ \infty_2)$ and $(51,\ 9,\ 8,\ 50,\ \infty_0)$, 5-tuples $(1,\ 43,\ 42,\ 0,\ \infty_2)$ and $(3,\ 45,\ 44,\ 2,\ \infty_0)$ must be changed to $(1,\ 43,\ 0,\ 42,\ \infty_2)$ and  $(3,\ 45,\ 2,\ 44,\ \infty_0)$, respectively. The other base blocks are developed by $(+2\ {\rm mod}\ 84)$. $\infty_0$ in the 4th row is replaced by $\infty_x$ when adding any values congruent to $2x\ ({\rm mod}\ 6)$. Every base block in 14th row forms 14 blocks and also blocks in this row form a cyclical trade of volume 3. Eventually construct a $2$-$(11,5,1)$DD on points $\infty_0,\dots,\infty_{10}$.
$$
\begin{array}{l}
(5,\ 47,\ \infty_0,\ 42,\ 0),\ \ \ \ (47,\ 5,\ 4,\ 46,\ \infty_1)\\
(7,\ 49,\ \infty_1,\ 44,\ 2),\ \ \ \ (49,\ 7,\ 6,\ 48,\ \infty_2)\\
(9,\ 51,\ \infty_2,\ 46,\ 4),\ \ \ \ (51,\ 9,\ 8,\ 50,\ \infty_0)\\
(\infty_0,\ 34,\ 79,\ 62,\ 3),\ \ (30,\ 67,\ 34,\ \infty_0,\ 47)\\
(83,\ 69,\ 18,\ 81,\ 13),\ \ (31,\ 13,\ 81,\ 60,\ 23)\\
(1,\ 47,\ 70,\ 2,\ 21),\ \ \ \ \ (21,\ 2,\ 53,\ 46,\ 57)\\
(26,\ 56,\ 20,\ 44,\ 59),\ \ (74,\ 24,\ 67,\ 56,\ 26)\\
(32,\ 46,\ 0,\ 8,\ 58),\ \ \ \ \ (8,\ 0,\ 64,\ 9,\ 70)\\
(61,\ 78,\ 25,\ 1,\ 63),\ \ \ \ (33,\ 67,\ 63,\ 1,\ 36)\\
(81,\ 18,\ \infty_3,\ 2,\ 75),\ \ \ (4,\ 31,\ \infty_4,\ 75,\ 2)\\
(13,\ 20,\ \infty_5,\ 10,\ 1),\ \ \ (1,\ 10,\ \infty_6,\ 50,\ 7)\\
(74,\ 57,\ \infty_7,\ 13,\ 0)\\
(0,\ 5,\ \infty_8,\ 59,\ 20)\\
(12,\ 61,\ \infty_9,\ 33,\ 8),\ \ (40,\ 5,\ \infty_9,\ 61,\ 36),\ \ (68,\ 33,\ \infty_9,\ 5,\ 64)\\
(71,\ 62,\ \infty_{10},\ 0,\ 61).\\
\end{array}
$$

For $v=115$, develop the following base blocks as follows, (see $\cite{pentagon}$).\\
\noindent First in the following blocks  $\infty_0$  is replaced by $\infty_1$ when adding 2 $({\rm mod}\ 104)$.
$$
\begin{array}{l}
(9,\ 61,\ 56,\ 4,\ \infty_0),\ \ \ \ \ \ (61,\ 9,\ 36,\ 88,\ \infty_1)\\
(13,\ 65,\ 60,\ 8,\ \infty_0),\ \ \ \ (65,\ 13,\ 40,\ 92,\ \infty_1)\\
(17,\ 69,\ 64,\ 12,\ \infty_0),\ \ \ (69,\ 17,\ 44,\ 96,\ \infty_1)\\
(21,\ 73,\ 68,\ 16,\ \infty_0),\ \ \ (73,\ 21,\ 48,\ 100,\ \infty_1)\\
(25,\ 77,\ 72,\ 20,\ \infty_0),\ \ \ (77,\ 25,\ 52,\ 0,\ \infty_1)\\
(29,\ 81,\ 76,\ 24,\ \infty_0),\ \ \ (81,\ 29,\ 4,\ 56,\ \infty_1)\\
(33,\ 85,\ 80,\ 28,\ \infty_0),\ \ \ (85,\ 33,\ 8,\ 60,\ \infty_1)\\
(37,\ 89,\ 84,\ 32,\ \infty_0),\ \ \ (89,\ 37,\ 12,\ 64,\ \infty_1)\\
(41,\ 93,\ 88,\ 36,\ \infty_0),\ \ \ (93,\ 41,\ 16,\ 68,\ \infty_1)\\
(45,\ 97,\ 92,\ 40,\ \infty_0),\ \ \ (97,\ 45,\ 20,\ 72,\ \infty_1)\\
(49,\ 101,\ 96,\ 44,\ \infty_0),\ \ (101,\ 49,\ 24,\ 76,\ \infty_1)\\
(53,\ 1,\ 100,\ 48,\ \infty_0),\ \ \ (1,\ 53,\ 28,\ 80,\ \infty_1)\\
(57,\ 5,\ 0,\ 52,\ \infty_0),\ \ \ \ \ \ (5,\ 57,\ 32,\ 84,\ \infty_1)\\
\end{array}
$$
Then develop the following base blocks $(+2\ {\rm mod}\ 104)$, in the first row  $\infty_0$  is replaced by $\infty_1$ when adding any values congruent to $2\ ({\rm mod}\ 4)$. The base block in the 11th row contains four disjoint cyclical  trades of volume 13, and also about the base block in the 13th row.
$$
\begin{array}{l}
(\infty_0,\ 47,\ 2,\ 4,\ 41),\ \ \ (4,\ 2,\ 44,\ 26,\ 49)\\
(14,\ 27,\ 0,\ 15,\ 11),\ \ \ (8,\ 71,\ 73,\ 11,\ 15)\\
(1,\ 51,\ 71,\ 81,\ 34),\ \ \ (91,\ 88,\ 51,\ 1,\ 9)\\
(0,\ 51,\ 49,\ 29,\ 87),\ \ \ (51,\ 0,\ 91,\ 77,\ 68)\\
(90,\ 0,\ 30,\ 20,\ 99),\ \ \ (7,\ 46,\ 30,\ 0,\ 26)\\
(88,\ 73,\ 38,\ 0,\ 32),\ \ \ (54,\ 66,\ 58,\ 32,\ 0)\\
(91,\ 55,\ 0,\ 61,\ 73),\ \ \ (89,\ 94,\ \infty_2,\ 61,\ 0)\\
(0,\ 53,\ \infty_3,\ 85\ 56),\ \ \ (97,\ 98,\ 53,\ 0,\ 27)\\
(31,\ 68,\ 0\ 21\ 76),\ \ \ \ \ (21,\ 0,\ \infty_4,\ 93,\ 92)\\
(86,\ 63,\ \infty_5,\ 79,\ 0)\\
(40,\ 19,\ \infty_6,\ 11,\ 0)\\
(75,\ 0,\ \infty_7,\ 62,\ 103)\\
(0,\ 19,\ \infty_8,\ 43,\ 28)\\
(69,\ 0,\ \infty_9,\ 80,\ 31)\\
(83,\ 44,\ \infty_{10},\ 0,\ 35).\\
\end{array}
$$
Eventually construct a $2$-$(11,5,1)$DD on points $\infty_0,\dots,\infty_{10}$.

For $v=135$, develop the following base blocks as follows, (see $\cite{pentagon}$).
$$
\begin{array}{l}
(\infty_{1},\ 4,\ 89,\ 18,\ 69),\ \ \ \ \ \ (29,\ 21,\ 18,\ 89,\ 53)\ \ \ \ \ (+6\ {\rm mod}\ 114)\ {\rm form}\ 6 \ {\rm blocks}\\
(\infty_{1},\ 26,\ 111,\ 40,\ 91),\ \ \ (51,\ 43,\ 40,\ 111,\ 75)\ \ \ (+6\ {\rm mod}\ 114)\ {\rm form}\ 8 \ {\rm blocks}    \\
(\infty_{1},\ 60,\ 31,\ 74,\ 11),\ \ \ \ (85,\ 77,\ 74,\ 31,\ 109)\ \ \ (+6\ {\rm mod}\ 114)\ {\rm form}\ 8 \ {\rm blocks}\\
(2,\ 87,\ 16,\ 67,\ \infty_{1}),\ \ \ \ \ \ (27,\ 19,\ 16,\ 87,\ 51)\ \ \ \ (+6\ {\rm mod}\ 114)\ {\rm form}\ 4 \ {\rm blocks}\\
(12,\ 97,\ 26,\ 77,\ \infty_{1}),\ \ \ \ (37,\ 29,\ 26,\ 97,\ 61)\ \ \ \ (+6\ {\rm mod}\ 114)\ {\rm form}\ 8 \ {\rm blocks}\\
(46,\ 17,\ 60,\ 111,\ \infty_{1}),\ \ \ (71,\ 63,\ 60,\ 17,\ 95)\ \ \ \ (+6\ {\rm mod}\ 114)\ {\rm form}\ 8 \ {\rm blocks}\\
(80,\ 51,\ 94,\ 31,\ \infty_{1}),\ \ \ \ (105,\ 97,\ 94,\ 51,\ 15)\ \ \ (+6\ {\rm mod}\ 114)\ {\rm form}\ 6 \ {\rm blocks}\\
(0,\ \infty_{0},\ 11,\ 37,\ 98),\ \ \ \ \ (37,\ 11,\ 33,\ 75,\ 95)\ \ \ \ (+10\ {\rm mod}\ 114)\ {\rm form}\ 11 \ {\rm blocks}\\
(112,\ \infty_{0},\ 9,\ 35,\ 96),\ \ \ (35,\ 9,\ 31,\ 73,\ 93)\ \ \ \ \ \ (+10\ {\rm mod}\ 114)\ {\rm form}\ 9 \ {\rm blocks}\\
(110,\ \infty_{0},\ 7,\ 33,\ 94),\ \ \ (33,\ 7,\ 29,\ 71,\ 91)\ \ \ \ \ \ (+10\ {\rm mod}\ 114)\ {\rm form}\ 3 \ {\rm blocks}\\
(88,\ 99,\ 11,\ \infty_{0},\ 72),\ \ \ (11,\ 99,\ 7,\ 49,\ 69)\ \ \ \ \ \ (+10\ {\rm mod}\ 114)\ {\rm form}\ 11 \ {\rm blocks}\\
(86,\ 97,\ 9,\ \infty_{0},\ 70),\ \ \ \ (9,\ 97,\ 5,\ 47,\ 67)\ \ \ \ \ \ \ \ (+10\ {\rm mod}\ 114)\ {\rm form}\ 9 \ {\rm blocks}\\
(62,\ \infty_{0},\ 73,\ 99,\ 46),\ \ \ (99,\ 73,\ 95,\ 23,\ 43)\ \ \ \ (+10\ {\rm mod}\ 114)\ {\rm form}\ 5 \ {\rm blocks}\\
(36,\ 47,\ 73,\ \infty_{0},\ 20),\ \ \ (73,\ 47,\ 69,\ 111,\ 17)\ \ \ (+10\ {\rm mod}\ 114)\ {\rm form}\ 5 \ {\rm blocks}\\
(84,\ 95,\ 7,\ \infty_{0},\ 68),\ \ \ \ \ (7,\ 95,\ 3,\ 45,\ 65)\ \ \ \ \ \ \ (+10\ {\rm mod}\ 114)\ {\rm form}\ 3 \ {\rm blocks}\\
\end{array}
$$
$$
\begin{array}{l}
(61,\ 0,\ \infty_{4},\ 64,\ 63),\ \ \ \ (0,\ 61,\ \infty_{3},\ 112,\ 101),\ \ \ (101,\ 112,\ \infty_{2},\ 50,\ 57)\\
(103,\ 101,\ 17,\ 1,\ 35),\ \ \ (25,\ 19,\ 35,\ 1,\ 48)\\
(4,\ 96,\ 28,\ 0,\ 60),\ \ \ \ \ \ \ (112,\ 74,\ 104,\ 60,\ 0)\\
(10,\ 60,\ 1,\ 82,\ 76),\ \ \ \ \ \ (33,\ 82,\ 1,\ 70,\ 90)\\
(110,\ 35,\ 40,\ 41,\ 0),\ \ \ \ (0,\ 41,\ \infty_{20},\ 97,\ 42)\\
(34,\ 10,\ 59,\ 0,\ 31),\ \ \ \ \ \ (31,\ 0,\ \infty_{19},\ 48,\ 67)\\
(88,\ 105,\ \infty_{18},\ 83,\ 0),\ \ (80,\ 95,\ \infty_{17},\ 0,\ 83)\\
(5,\ 26,\ \infty_{16},\ 0,\ 79),\ \ \ \ \ (102,\ 75,\ \infty_{15},\ 79,\ 0)\\
(13,\ 46,\ \infty_{14},\ 89,\ 0),\ \ \ \ (23,\ 108,\ \infty_{13},\ 0,\ 89)\\
(15,\ 56,\ \infty_{12},\ 0,\ 23)\\
(76,\ 19,\ \infty_{11},\ 9,\ 0)\\
(32,\ 107,\ \infty_{10},\ 37,\ 0)\\
(0,\ 77,\ \infty_{9},\ 13,\ 28)\\
(87,\ 18,\ \infty_{8},\ 0,\ 27)\\
(43,\ 0,\ \infty_{7},\ 36,\ 29)\\
(1,\ 0,\ \infty_{6},\ 94,\ 73)\\
(17,\ 0,\ \infty_{5},\ 84,\ 69)\\
\end{array}
$$
Any of base blocks from 16th row to the last row is developed by  $(+2\ {\rm mod}\ 114)$.\\
\noindent In base block $(25,\ 19,\ 35,\ 1,\ 48)$, 5-tuples $(71,\ 65,\ 81,\ 47,\ 94)$, $(77,\ 71,\ 87,\ 53,\ 100)$, $(61,\ 55,\ 71,\ 37,\ 84)$ and $(55,\ 49,\ 65,\ 31,\ 78)$ must be changed to $(71,\ 81,\ 47,\ 94,\ 65)$, $(77,\ 87,\ 53,\ 100,\ 71)$, $(61,\ 55,\ 37,\ 71,\ 84)$ and $(55,\ 49,\ 31,\ 65,\ 78)$, respectively. And also in base block $(103,\ 101,\ 17,\ 1,\ 35)$, 5-tuples $(53,\ 51,\ 81,\ 65,\ 99)$, $(59,\ 57,\ 87,\ 71,\ 105)$, $(25,\ 23,\ 53,\ 37,\ 71)$ and $(19,\ 17,\ 47,\ 31,\ 65)$ must be changed to $(53,\ 51,\ 65,\ 81,\ 99)$, $(59,\ 57,\ 71,\ 87,\ 105)$, $(25,\ 23,\ 71,\ 53,\ 37)$ and $(19,\ 17,\ 65,\ 47,\ 31)$, respectively.\\
\noindent Any of the last four base blocks above contains three disjoint cyclical trades of volume 19. And base blocks  in the 16th row contain five disjoint cyclical trades of volume 15 (for example see the following cyclical trade) and 48 directed trades of volume 2.
$$
\begin{array}{l}
(65,\ 4,\ \infty_{4},\ 68,\ 67),\ \ \ \ \ \ (4,\ 65,\ \infty_{3},\
2,\ 105),\ \ \ \ \ (105,\ 2,\ \infty_{2},\ 54,\ 61), \ \ \  (43,\
54,\ \infty_{2},\ 106,\ 113),\ \\(95,\ 106,\ \infty_{2},\ 44,\
51),\ \ \ (33,\ 44,\ \infty_{2},\ 96,\ 103),\ \ (85,\ 96,\
\infty_{2},\ 34,\ 41),\ \ \ (23,\ 34,\ \infty_{2},\ 86,\ 93),\ \\
(75,\ 86,\ \infty_{2},\ 24,\ 31),\ \ \ \ (13,\ 24,\ \infty_{2},\
76,\ 83),\ \ \ \ (65,\ 76,\ \infty_{2},\ 14,\ 21),\ \ \ (3,\ 14,\
\infty_{2},\ 66,\ 73),\ \\(55,\ 66,\ \infty_{2},\ 4,\ 11),\ \ \ \
\ (68,\ 15,\ \infty_{3},\ 66,\ 55),\ \ \ \ (15,\ 68,\
\infty_{4},\ 18,\ 17)\
\end{array}
$$
Eventually It is sufficient to construct a $2$-$(21,5,1)$DD on points $\infty_0,\dots,\infty_{20}$ and to take the following blocks.
$$
\begin{array}{l}
(\infty_{1},\ 106,\ 77,\ 6,\ 57),\ \ \ \ (17,\ 9,\ 6,\ 77,\ 41)\\
(108,\ \infty_{1},\ 79,\ 8,\ 59),\ \ \ \ (19,\ 11,\ 8,\ 79,\ 43)\\
(\infty_{1},\  112,\ 83,\ 12,\ 63),\ \ (23,\ 15,\ 12,\ 83,\ 47)\\
(0,\ \infty_{1},\ 85,\ 14,\ 65),\ \ \ \ \ (25,\ 17,\ 14,\ 85,\ 49)\\
(65,\ \infty_{1},\ 94,\ 108,\ 45),\ \ (71,\ \infty_{1},\ 100,\ 0,\ 51),\ \ (6,\ 91,\ \infty_{1},\ 20,\ 71)\\
(40,\ 11,\ \infty_{1},\ 54,\ 105),\ \ (74,\ 45,\ \infty_{1},\ 88,\ 25),\ \ (26,\ 37,\ \infty_{0},\ 63,\ 10)\\
(5,\ 111,\ 108,\ 65,\ 29),\ \ \ (11,\ 3,\ 0,\ 71,\ 35),\ \ \ \ \ \ (31,\ 23,\ 20,\ 91,\ 55)\\
(65,\ 57,\ 54,\ 11,\ 89),\ \ \ \ (99,\ 91,\ 88,\ 45,\ 9),\ \ \ \ \ (63,\ 37,\ 59,\ 101,\ 7).\\
\end{array}
$$
}
\end{proof}
\begin{theorem}
For all $v\equiv15\ ({\rm mod}\ 20);\ v>15$ there exists a super-simple $2$-$(v,5,1)$DD with $f\geq\frac{1}{2}$.
\end{theorem}
\begin{proof}
{For $v=35,\ 55,\ 75,\ 95,\ 115,\ 135$, see the previous Lemma. In Lemma 3.4 it is shown that for all $n$ except $n\in \{2,11,17,23,32\}$ there exists a $\{5,6\}$-GDD of type $5^{4n+1}x^{1}$ where $x\leq 5n$. Give weight 6 to all points in the first $4n+1$ groups and weight 0, 6 or 8 to the points in the last group then use the super-simple DGDD of type $(6^{5} 8^{1})$, super-simple DGDD of type $(6^{5})$, super-simple DGDD of type $(6^{6})$. If we give weight 6 to all points in the first $4n+1$ groups and weight 8 to three points or weight 8 to all points in the last group of $\{5,6\}$-GDD of type $5^{x}k^{1}$ then we can form a super-simple DGDD of type $(30^{4n+1} 24^{1})$ or super-simple DGDD of type $(30^{4n+1} 40n^{1})$ respectively. Generally for all $n$ except $n\in \{2,11,17,23,32\}$ there exists a super-simple DGDD of type $(30^{4n+1} g^{1})$ where $24\leq g\leq 40n$. There exists a super-simple $2$-$(31,5,1)$DD. Therefor if there is a super-simple $2$-$(g+1,5,1)$DD where $g+1\equiv5\ ({\rm mod}\ 20)$, we can add a new point to this super-simple DGDD of type $(30^{4n+1} g^{1})$ and construct a super-simple $2$-$(120n+30+g+1,5,1)$DD.\\
\noindent It remains the values $\{155,195,\dots,395,515\}$.\\
\noindent $\bullet\ \ \ v=155$: first we form a super-simple DGDD of type $(31^5)$ as follows.

\noindent For constructing of super-simple DGDD of type $(31^5)$, we can use the following result which is obtained in $\cite{pentagon}$.

\noindent  %
\begin{lemma}~{\rm \cite{pentagon}}
If $q$ is a odd prime power, then there exists a super-simple TD$_2(q,q)$ which is the union of two TD$(q,q)$s.
\end{lemma}
\noindent Now applying Lemma 3.5 with $k=5$ and $q=31$ and using previous Lemma we can obtain a super-simple $(5,2)$-GDD of type $(31^5)$ which is the union of two $TD(5,31)$s. Now by a suitable arrangement in these two $TD(5,31)$s we  can construct a super-simple DGDDs of type $(31^5)$ with $f\geq\frac{1}{2}$. It's sufficient to construct a super-simple $2$-$(31,5,1)$DD with $f\geq\frac{1}{2}$ on the points of each group.\\
\noindent $\bullet\ \ \ v=195$: there exists a TD$(6,8)$ $\cite{PBD}$, add a point $y$ to the groups, then delete a different point so as to form a $\{6,9\}$-GDD of type $5^{8} 8$. Give weight 4 to all points in the groups of size 5 and points in the group of size 8 except added point $y$ then give weight 6 to $y$ $\cite{pentagon}$. Since $y$ is in the blocks of size 9 so we replace any of them by a super-simple DGDD of type $(4^8\ 6^1)$ with $f\geq\frac{1}{2}$ and for blocks of size 6 use a super simple DGDD of type $(4^6)$ finally use a super-simple $2$-$(21, 5, 1)$DD and a super-simple $2$-$(35,5,1)$DD.\\
\noindent $\bullet\ \ \ v=215$: like in the $v=195$  first use a TD$(6,9)$ $\cite{PBD}$ to construct a $\{6,10\}$-GDD of type $5^{9} 9$ then delete one point from the last group (not $y$) and give weight 4 to all other points. The point $y$ is in the blocks of size 10 so we replace any of them by a super-simple DGDD of type $(4^{10})$ and for blocks of size 6 use a super simple DGDD of type $(4^5)$, $(4^6)$. This yields a super-simple DGDD of type $(20^9\ 32)$, finally add three new points to groups, then on each of 9 groups of size 20 together with the three new points construct a super-simple DGDD of type $(1^{20}\ 3^1)$ and on the last group together with the three new points construct a super-simple $2$-$(35,5,1)$DD.\\
\noindent  $\bullet\ \ \ v=235$: in $\cite{pentagon2004}$ it is shown that there exists a super-simple $2$-$(235,5,1)$ design, by the same way we can form a super-simple $2$-$(235,5,1)$DD with $f\geq\frac{1}{2}$.\\
\noindent $\bullet\ \ \ v=255$: if we remove one point of a TD$(6,11)$ $\cite{PBD}$, then we can construct a $\{6,11\}$-GDD of type $5^{11} 10$. Give weight 4 to all points in the groups of size 5 and in the last group give weight 4 to 8 points and use super-simple DGDDs of type $(4^5)$, $(4^6)$ and $(4^{11})$. This yields a super-simple DGDD of type $(20^{11}\ 32)$, finally add three new points to groups, then on each of 11 groups of size 20 together with the three new points construct a super-simple DGDD of type $(1^{20}\ 3^1)$ and on the last group together with the three new points construct a super-simple $2$-$(35,5,1)$DD.\\
\noindent $\bullet\ \ \ v\in \{275,315,375\}$: there exist a TD$(5,11)$, a $TD(7,9)$ and a $TD(5,15)$ $\cite{PBD}$, give weight 5 to the points in all groups to construct a super-simple $2$-$(275,5,1)$DD, a super-simple $2$-$(315,5,1)$DD and a super-simple $2$-$(375,5,1)$DD with $f\geq\frac{1}{2}$, respectively.\\
\noindent $\bullet\ \ \ v\in \{295,335\}$: in Lemma 3.3 it is shown that there exist a $\{5,6\}$-GDD of type $9^{5} 3$ and a $\{5,6\}$-GDD of type $9^{5} 8$. Give weight 6 to the points in the first five groups and  weight 8 to the points in the last groups and add an extra point to each group to construct a super-simple $2$-$(295,5,1)$DD and a super-simple $2$-$(335,5,1)$DD with $f\geq\frac{1}{2}$, respectively.\\
\noindent $\bullet\ \ \ v=355$: if we delete one point of a TD$(6,16)$ $\cite{PBD}$, then we can construct a $\{6,16\}$-$GDD$ of type $5^{16} 15$. Give weight 4 to all points in the groups of size 5 and in the last group give weight 4 to 8 points and use super-simple DGDDs of type $(4^5)$, $(4^6)$ and $(4^{16})$. This yields a super-simple DGDD of type $(20^{16}\ 32)$, finally add three new points to groups, then on each of 16 groups of size 20 together with the three new points construct a super-simple DGDD of type $(1^{20}\ 3^1)$ and on the last group together with the three new points construct a super-simple $2$-$(35,5,1)$DD.\\
\noindent $\bullet\ \ \ v=515$: in Lemma 3.4 it is shown that there
exists a $\{5,6\}$-$GDD$ of type $5^{21}23^{1}$. Give weight 4 to
the points in all groups to construct a super-simple DGDD of type
$(20^{21}\ 92)$, then add three new points to groups, then on each
of 21 groups of size 20 together with the three new points construct
a super-simple $DGDD$ of type $(1^{20}\ 3^1)$ and on the last group
together with the three new points construct a super-simple
$2$-$(95,5,1)$DD. }
\end{proof}
In $\cite{pentagon2004}$ it is shown that there exists a
super-simple $2$-$(91,5,2)$ design, by the same way we can form a
super-simple $2$-$(91,5,1)$DD.

Thus we have proved the following theorems.
\begin{theorem}
For all $v\equiv1,5\ ({\rm mod}\ 10)$, except $v=5,15$ there exists a super-simple $2$-$(v,5,1)$DD.
\end{theorem}
\begin{theorem}
For all $v\equiv1,5\ ({\rm mod}\ 10)$, except $v=5,15$ and, except possibly $v=11,91$, there exists a super-simple $2$-$(v,5,1)$DD with $f\geq\frac{1}{2}$.
\end{theorem}
%

%


\begin{thebibliography}{1}


\bibitem{pentagon}
 R. J. R. Abel and F. E. Bennett, Super-simple Steiner pentagon systems,  {\it Discrete Appl. Math.}, {\bf 156}
  (2008), 780--793.


 \bibitem{HSPS}
 R. J. R. Abel, F. E. Bennett and G. Ge, Super-simple holey Steiner pentagon systems and related designs, {\it J. Combin. Des.}, {\bf 16}
  (2008), 301--328.

\bibitem{farzane}
 F. Amirzade and N. Soltankhah, Smallest
defining sets of super-simple $2-(v,4,1)$ directed designs, {\it
Utilitas Math.}, (to appear).

 \bibitem{PBD}
 F. E. Bennett, Y. Chang, G. Ge and M. Greig, Existence of $(v,\{5,w^*\},1)-PBD$s, {\it Discrete Math.}, {\bf 279}
  (2004), 61--105.

\bibitem{Chen and Wei2007}
 K. Chen and R. Wei, Super-simple $(v,5,4)$ designs, {\it Discrete Appl. Math.}, {\bf 155}
  (2007), 904--913.

\bibitem{Chen and Wei2006}
 K. Chen and R. Wei, Super-simple $(v,5,5)$ designs, {\it Des. Codes Cryptogr.}, {\bf 39}
  (2006), 173--187.


\bibitem{Quinn}
 M. J. Grannell, T. S. Griggs and K. A. S. Quinn, Smallest
defining sets of directed triple systems, {\it Discrete Math.},
{\bf 309} (2009), 4810--4818.


\bibitem{pentagon2004}
 H.-D. O. F. Gronau, D. L. Kreher and A. C. H. Ling, Super-simple $2-(v,5,2)$ designs,
{\it Discrete Appl. Math.}, {\bf 138} (2004), 65--77.


\bibitem {MMS}
 E. S. Mahmoodian, N. Soltankhah and A. P. Street, On
defining sets of directed designs, {\it Australas. J. Combin.},
{\bf 19} (1999), 179--190.


\bibitem{mullin}
 R. C. Mullin and H.-D. O. F. Gronau, {\it PBDs and GDDs: the
basics}, in The CRC Handbook of Combinatorial Designs, second
edition (ed. C. J. Colbourn and J. H. Dinitz), CRC Press, (2007), 231--236.



\bibitem{trade}
 N. Soltankhah, On directed trades, {\it Australas. J. Combin.},
{\bf 11} (1995), 59--66.



\bibitem{DBIBD}
 D. J. Street and W. H. Wilson, On directed balanced incomplete block designs with block size
five, {\it Utilitas Math.}, {\bf 18} (1980), 27--34.



\end{thebibliography}
\end{document}